\documentclass[a4paper,11pt]{article}    
      
\usepackage{multirow}
\usepackage{mathtools}
\usepackage{subcaption}
\usepackage{amsfonts,amssymb,amsmath,amsthm,latexsym,bbm}
\usepackage{enumerate}
\usepackage{epsf,epsfig}
\usepackage{xcolor,colortbl,color}
\usepackage{graphicx,graphics}
\usepackage{caption}
\usepackage[utf8]{inputenc}
\usepackage{tikz}
\usepackage{stmaryrd }
\usepackage{epstopdf}
\usepackage[english]{babel}
\makeatletter \@addtoreset{equation}{section} \makeatother
\makeatletter \@addtoreset{enunciato}{section} \makeatother
\newcounter{enunciato}[section]

\newtheorem{ittheorem}{Theorem}
\newtheorem{itlemma}{Lemma}
\newtheorem{itproposition}{Proposition}
\newtheorem{itdefinition}{Definition}
\newtheorem{itremark}{Remark}
\newtheorem{itclaim}{Claim}
\newtheorem{itfact}{Fact}
\newtheorem{itconjecture}{Conjecture}
\newtheorem{itcorollary}{Corollary}

\newenvironment{theorem}{\addtocounter{enunciato}{1}
\begin{ittheorem}}{\end{ittheorem}}
\newenvironment{lemma}{\addtocounter{enunciato}{1}
\begin{itlemma}}{\end{itlemma}}
\newenvironment{proposition}{\addtocounter{enunciato}{1}
\begin{itproposition}}{\end{itproposition}}
\newenvironment{definition}{\addtocounter{enunciato}{1}
\begin{itdefinition}}{\end{itdefinition}}
\newenvironment{remark}{\addtocounter{enunciato}{1}
\begin{itremark}}{\end{itremark}}

\newenvironment{conjecture}{\addtocounter{enunciato}{1}
\begin{itconjecture}}{\end{itconjecture}}
\newenvironment{corollary}{\addtocounter{enunciato}{1}
\begin{itcorollary}}{\end{itcorollary}}

\newcommand{\be}[1]{\begin{equation}\label{#1}}
\newcommand{\ee}{\end{equation}}
\newcommand{\bl}[1]{\begin{lemma}\label{#1}}
\newcommand{\el}{\end{lemma}}
\newcommand{\br}[1]{\begin{remark}\label{#1}}
\newcommand{\er}{\end{remark}}
\newcommand{\bt}[1]{\begin{theorem}\label{#1}}
\newcommand{\et}{\end{theorem}}
\newcommand{\bd}[1]{\begin{definition}\label{#1}}
\newcommand{\ed}{\end{definition}}
\newcommand{\bp}[1]{\begin{proposition}\label{#1}}
\newcommand{\ep}{\end{proposition}}
\newcommand{\bc}[1]{\begin{corollary}\label{#1}}

\newcommand{\bcj}[1]{\begin{conjecture}\label{#1}}
\newcommand{\ecj}{\end{conjecture}}

\def\<{\langle}
\def\>{\rangle}
\def\chv#1{\{\,#1\,\}}

\def\abs#1{\left\vert #1 \right\vert} 
\def\prt#1{\left( #1 \right)} 
\def\crt#1{\left[ #1 \right]} 
\def\crl#1{\left\{ #1 \right\}} 
\def\Ind#1{\mathbbm{1}_{#1}}  
\def\teto#1{\left\lceil #1 \right\rceil} 
\def\piso#1{\left\lfloor #1 \right\rfloor} 

\def\eatspace#1{\relax}
\def\unskipit{\expandafter\eatspace}

\newcommand{\gga}{\gamma}

\newcommand{\gd}{\delta}
\newcommand{\gD}{\Delta}
\newcommand{\gep}{\varepsilon} 

\newcommand{\gt}{\theta}

\newcommand{\gl}{\lambda}

\newcommand{\gs}{\sigma}

\newcommand{\go}{\omega}
\newcommand{\gO}{\Omega}
\newcommand{\gi}{\iota}

\newcommand{\mc}[1]{{\mathcal #1}}

\newcommand{\bb}[1]{{\mathbb #1}}

\newcommand{\R}{\ensuremath{\mathbb{R}}}

\newcommand{\Z}{\ensuremath{\mathbb{Z}}}
\newcommand{\N}{\ensuremath{\mathbb{N}}}
\newcommand{\PP}{\ensuremath{\mathbb{P}}}

\newcommand{\EE}{\ensuremath{\mathbb{E}}}
\newcommand{\dd}{\ensuremath{\mathrm{d}}}

\def\crt#1{\left[ #1\right]} 
\def\prt#1{\left( #1\right)} 
\def\chv#1{\{\,#1\,\}} 
\def\Ind#1{ \mathbbm{1}_{#1}} 
\def\abs#1{\left\vert #1\right\vert}

\begin{document}    


\title{Random walk in cooling random environment:\\
ergodic limits and concentration inequalities}

\author{\renewcommand{\thefootnote}{\arabic{footnote}}
Luca Avena,\,\,
Yuki Chino,\,\,
Conrado da Costa,\,\,
Frank den Hollander\,\footnotemark[1]}

\date{\today}

\footnotetext[1]{Mathematical Institute, Leiden University, 
P.O.\ Box 9512, 2300 RA Leiden, The Netherlands}


\maketitle


\begin{abstract}
In previous work by Avena and den Hollander~\cite{AdH17}, a model of a one-dimensional 
random walk in a dynamic random environment was proposed where the random environment 
is resampled from a given law along a growing sequence of deterministic times. In the regime 
where the increments of the resampling times diverge, which is referred to as the cooling 
regime, a weak law of large numbers and certain fluctuation properties were derived under 
the annealed measure. In the present paper we show that a strong law of large numbers 
and a quenched large deviation principle hold as well. In the cooling regime, the random walk 
can be represented as a sum of independent variables, distributed as the increments of a
random walk in a static random environment over increasing periods of time. Our proofs require 
suitable multi-layer decompositions of sums of random variables controlled by moments bounds 
and concentration estimates. Along the way we derive two results of independent interest, namely, 
a concentration inequality for the cumulants of the displacement in the static random environment 
and an ergodic theorem that deals with limits of sums of triangular arrays representing the 
structure of the cooling regime. We close by discussing our present understanding of homogenisation 
effects as a function of the speed of divergence of the increments of the resampling times.  

\\
{\it Acknowledgment:} The research in this paper was supported through NWO 
Gravitation Grant NETWORKS-024.002.003. 
\end{abstract}

\newpage


\section{Introduction, main results and discussion} 
\label{Intro} 

Random walk in random environment is a model for a particle moving in an inhomogeneous 
potential. When the random environment is \emph{static}  this model exhibits striking features. 
Namely, there are regions where the random walk remains \emph{trapped} for a long time. 
The presence of these traps leads to a local slow-down of the random walk in comparison to 
a homogeneous random walk, and may result in anomalous scaling, especially in low dimensions.  
At present, these slow-down phenomena have been fully understood only in dimension one 
(see Zeitouni~\cite{Z04}, and references therein). 

The situation where the random environment is \emph{dynamic}  has seen major progress 
in the last ten years. While the random environment evolves over time, it remains inhomogeneous but 
dissolves existing traps and creates new traps. Depending on the choice of the dynamics, the random 
walk behaviour can either be similar to that in the static model or be similar to that in the homogeneous 
model. Up to now, most dynamic models require strong space-time mixing conditions, guaranteeing 
negligible trapping effects and resulting in scaling properties similar to those of a homogeneous random 
walk (see Avena, Blondel and Faggionato~\cite{ABF17}, and references therein).

In Avena and den Hollander~\cite{AdH17}, a new random walk model was introduced, called 
\emph{Random Walk in Cooling Random Environment} (RWCRE). This has a dynamic random 
environment, but differs from other dynamic models in that it allows for an explicit control of the 
time mixing in the environment. Namely, at time zero an i.i.d.\ random environment is generated, 
and this is \emph{fully resampled} along an increasing sequence of deterministic times. If the 
resampling times increase rapidly enough, then we expect to see a behaviour close to that of 
the static model. Conversely, if the resampling times increase slowly enough, then we expect to 
see a behaviour that is close to the homogeneous model. Thus, RWCRE allows for different 
scenarios as a function of the speed of growth of the resampling times. The name ``cooling" is 
used because the static model is sometimes called ``frozen". 

In order to advance our understanding of RWCRE, we need to acquire detailed knowledge of 
fluctuations and large deviations for the classical one-dimensional Random Walk in Random 
Environment (RWRE). Some of this knowledge is available from the literature, but other parts 
are not and need to be developed along the way. A few preliminary results were proved in
Avena and den Hollander~\cite{AdH17} under the \emph{annealed} law. In the simplest scenario 
where the increments of the resampling times stay bounded, which is referred to as the 
\emph{no-cooling regime}, full homogenisation takes place, and both a classical Strong Law 
of Large Numbers (SLLN) and a classical Central Limit Theorem (CLT) hold. Moreover, it was 
shown that as soon as the increments of the resampling times diverge, which is referred to as 
the \emph{cooling regime}, a Weak Law of Large Numbers (WLLN) holds with an asymptotic 
speed that is the \emph{same} as for the corresponding RWRE~\cite[Theorem 1.5]{AdH17}. 
As far as fluctuations are concerned, for the case where the RWRE is in the so-called Sinai 
regime (recurrent, subdiffusive, non-standard limit law; see Sinai~\cite{S82}, Kesten \cite{K86}), 
it was shown that RWCRE exhibits Gaussian fluctuations with a scaling that \emph{depends} on 
the speed of divergence of the increments of the resampling times~\cite[Theorem 1.6]{AdH17}.
The proof of this fact requires that the convergence to the limit law for the corresponding RWRE 
is in $L^p$ for some $p>2$. In~\cite[Appendix C]{AdH17} it was shown that the convergence is 
in $L^p$ for all $p>0$.

In the present paper we pursue a more refined investigation of RWCRE. We focus on the cooling 
regime and aim for a deeper understanding of homogenisation effects. In particular, we derive a 
SLLN and a \emph{quenched} Large Deviation Principle (LDP), with a limiting speed and a rate 
function that are the \emph{same} as for the corresponding RWRE (Theorems~\ref{SLLN} 
and~\ref{thm:LDP} below). Both results are not unexpected, but at the same time are far from 
obvious. As we will see, they lead to some subtle surprises, which we discuss below. A crucial 
ingredient in both proofs is a general limit property we call \emph{cooling ergodic theorem}, 
which is needed to control certain variables representing the structure of the cooling regime 
(Theorem~\ref{CEThm} below). This theorem not only is a key tool in our proofs, it will also be 
useful to address other questions not investigated here. To prove the SLLN and the LDP we 
also need certain \emph{concentration inequalities} for the corresponding RWRE 
(Theorem~\ref{P:MB} below).

\medskip\noindent
{\bf Outline.}
In Section~\ref{RWRE} we define one-dimensional RWRE and recall some basic facts that are used 
throughout the paper. In Section~\ref{model} we define RWCRE. In Section~\ref{results} we state our 
four main theorems and provide some insight into their proofs. In Section~\ref{discussion} we discuss 
what is known about RWCRE, explain how the results derived so far relate to each other, and state 
a number of open problems. The remainder of the paper is devoted to the proofs: Section~\ref{CETp}
for the cooling ergodic theorem of RWCRE, Section~\ref{Cumu} for the concentration inequalities of 
RWRE, and Section~\ref{SLLDP} for the SLLN and the LDP of RWCRE.


\subsection{RWRE: some basic facts} 
\label{RWRE}

Throughout the paper we use the notation $\N_0 = \N \cup \{0\}$ with $\N = \{1,2,\dots\}$. The classical 
one-dimensional static model is defined as follows. Let $\omega=\{\omega(x)\colon\,x\in\Z\}$ be an 
i.i.d.\ sequence with probability distribution
\begin{equation} 
\label{alpha}
\mu = \alpha^{\Z}
\end{equation} 
for some probability distribution $\alpha$ on $(0,1)$.  We write $\langle\cdot\rangle$ to denote the 
expectation w.r.t.\ $\alpha$.

\begin{definition}[{\bf RWRE}]\label{RWREdef} 
{\rm Let  $\omega$ be an environment sampled from $\mu$. We call \emph{Random Walk in Random 
Environment} the Markov chain $Z = (Z_n)_{n\in\N_0}$ with state space $\Z$ and transition probabilities} 
\begin{equation} 
\label{Ker_om}
P^{\omega}(Z_{n+1} = x + e \mid Z_n = x) 
= \left\{
\begin{array}{ll}
\omega(x) &\mbox{ if } e = 1,\\  
1 - \omega(x) &\mbox{ if } e = - 1,
\end{array}
\right. 
\qquad n \in \N_0.
\end{equation}
{\rm We denote by} $P_x^{\omega}(\cdot)$ {\rm the \emph{quenched} law of the Markov chain identified by the 
transitions in~\eqref{Ker_om} starting from $x\in\Z$, and by \begin{equation*} P_{x}^{\mu}(\cdot) 
= \int P_x^{\omega}(\cdot)\,\mu(\dd\omega),
\end{equation*}
the corresponding \emph{annealed} law.}
\end{definition}

The understanding of one-dimensional RWRE is well developed, both under the quenched and the 
annealed law. For a general overview, we refer the reader to the lecture notes by Zeitouni~\cite{Z04}.
Here we collect some basic facts and definitions that will be needed throughout the paper. 

The asymptotic properties of RWRE are controlled by the distribution of the ratio of the transition 
probabilities to the left and to the right at the origin, i.e.,
\begin{equation} 
\label{rhodef}
\rho = \frac{1 - \omega(0)}{\omega(0)}.
\end{equation}
We will impose that the support of $\alpha$ is contained in an interval of the form $[\mathfrak{c},1-\mathfrak{c}]$
for some $\mathfrak{c}>0$. This corresponds to a \emph{uniform ellipticity} condition on $\mu$, meaning that
\begin{equation}
\label{UEllip}
\mu\big(\omega \colon\,0< \mathfrak{c}\leq\omega(x)\leq 1-\mathfrak{c}<1,\,\forall\,x\in\Z\big) = 1.
\end{equation}
Let $\rho_\text{max}$ and $\rho_{\text{min}}$ denote the maximum and minimum of $\rho$ over the support of 
$\alpha$.  We will further impose that
\begin{equation} 
\label{RWconfused}
\rho_{\text{min}} < 1 < \rho_\text{max}.
\end{equation}
The  inequalities in~\eqref{RWconfused} ensures that we are in the \emph{nested} situation, i.e., at some 
sites the random walk prefers to go to the right while at other sites it prefers to go to the left.

\begin{definition}[{\bf Basic environment distribution}] 
\label{D:mu}{\rm
We call a probability distribution $\mu$ on $(0,1)^\Z$  \emph{basic} (= i.i.d., uniform elliptic, nested) if 
\eqref{alpha},~\eqref{UEllip} and~\eqref{RWconfused} hold.}
\end{definition} 

 The following proposition due to Solomon~\cite{S75} characterises recurrence versus transience and 
asymptotic speed. To state the result in a simple form we may assume without loss of generality that
\begin{equation} 
\label{+0R}
\langle \log \rho \rangle \leq 0.
\end{equation}
The case where $\langle \log \rho \rangle > 0$ follows by a reflection argument. Indeed, define 
$\widetilde{\omega}$ by $\widetilde{\omega}(x) = 1-\omega(x)$,  $x \in \Z$. From~\eqref{Ker_om} 
we see that $P_0^\omega(-Z_n \in \cdot) = P_0^{\widetilde{\omega}}(Z_n \in \cdot)$. Therefore, 
statements for the left of the origin can be obtained from statements for the right of the origin in 
the reflected environment and so~\eqref{+0R} is assumed for convenience.

\begin{proposition}[{\bf Recurrence, transience and speed of RWRE}~\cite{S75}]
\label{prop:LLN}$\text{}$\\ 
Suppose that $\mu$ is basic and that~\eqref{+0R} holds. Then:
\begin{itemize} 
\item 
$Z$ is recurrent when $\langle \log \rho \rangle = 0$.
\item 
$Z$ is transient to the right when $\langle \log \rho \rangle <0$.
\item 
For $\mu \text{-a.e.} \,\omega$, $P_0^\omega $\text{-a.s.}, 
\begin{equation}\label{speed}
\lim_{n\to\infty} \frac{Z_n}{n} = v_\mu  =  \left\{\begin{array}{ll}
0, &\mbox{ if }  \langle\rho\rangle \geq 1,\\
\frac{1-\langle\rho\rangle}{1+\langle\rho\rangle} > 0 
, &\mbox{ if } \langle\rho\rangle < 1.
\end{array}
\right.
\end{equation}
\end{itemize}
\end{proposition}

\noindent 
The above proposition shows that the speed of RWRE is a deterministic function of $\mu$ 
(or of $\alpha$; recall~\eqref{alpha}). Note that for $\alpha$ such that $\langle \log \rho \rangle < 0$ 
and $\langle\rho\rangle \geq 1$, the random walk is transient to the right with zero speed. In this 
regime $Z$ diverges, but only sublinearly due to the presence of \emph{traps}, i.e., local regions 
of the environment pushing the random walk against its global drift.

Similar trapping effects give rise to other anomalous behaviour for fluctuations and large deviations.
In order to state the latter, we recall that a family of probability measures $(P_n)_{n\in\N}$ defined 
on the Borel sigma-algebra of a topological space $(\mathcal{S},\mc{T})$ is said to satisfy the LDP 
with rate $n$ and with rate function  $I\colon\,\mathcal{S} \to [0,\infty]$ when
\begin{equation}
\label{LDbds}
\begin{aligned}
\liminf_{n\to\infty} \frac{1}{n} \log P_n(\mathcal{O}) \geq - \inf_{x \in \mathcal{O}} I(x)
&\qquad \forall \;\mathcal{O} \subset \mathcal{S} \text { open},\\ 
\limsup_{n\to\infty} \frac{1}{n} \log P_n(\mathcal{C}) \leq - \inf_{x \in \mathcal{C}} I(x)
&\qquad \forall \;\mathcal{C} \subset \mathcal{S} \text { closed},
\end{aligned}
\end{equation}
$I$ has compact level sets and $I \not\equiv \infty$ (see e.g.\ den Hollander~\cite[Chapter III]{dH00}).
The following proposition due to Greven and den Hollander~\cite{GdH94} identifies the LDP for 
the empirical speed under the \emph{quenched} law. 

\begin{proposition}[{\bf Quenched LDP for RWRE displacements}~\cite{GdH94}]
\label{staticLDPque}
Suppose that $\mu$ is basic. Then, for  $\mu$-a.e.\ $\omega$, $(Z_n/n)_{n\in\N}$ under $P^\go_0$ 
satisfies the LDP on $\R$ with rate $n$ and with a convex and deterministic rate function $\mc{I}
=\mc{I}_\mu$.
\end{proposition}

\noindent 
See~\cite{GdH94} for a representation of $\mc{I}$ in terms of random continued fractions and Fig.~\ref{fig:Irf} for the qualitative behaviour of $\mc{I}$ on different regimes.

In the sequel we will need refined results about the cumulant generating function of $Z_n/n$. For that we 
need to introduce the hitting times to the right
\begin{equation}
\label{Htimes}
H_n = \inf \chv{m \in \N \colon Z_m = n}, \quad n \in \N,
\end{equation}
state the \emph{weak} LDP for $H_n/n$, which was derived in Comets, Gantert and Zeitouni~\cite{CGZ00}, 
and show its relation with the LDP for $Z_n/n$. See also den Hollander~\cite[Chapter VII]{dH00}. We recall 
that for the weak LDP the second line in~\eqref{LDbds} is only required to hold for compact sets, and the 
rate function is only required to be lower semi-continuous. 

\begin{proposition}[{\bf Quenched LDP for RWRE hitting times}~\cite{CGZ00}] 
\label{htRWRE} 
Suppose that $\mu$ is basic. Then, for  $\mu$-a.e.\ $\omega$, $(H_n/n)_{n\in\N}$ under $P^\go_0$ satisfies 
the weak LDP on $\R$ with rate $n$ and with a convex and deterministic weak rate function $\mc{J} =
\mc{J}_\mu$ given by (see Fig.~{\rm~\ref{rateJJ*}})
\begin{equation} 
\label{qRF_T}
\mc{J}(x) = \sup_{\gl\in\R} \big[\gl x - \mc{J}^*(\gl)\big], \qquad x \in \R,
\end{equation}
where
\begin{equation} 
\label{qCGF_T}
\mc{J}^*(\gl) = \lim_{n\to\infty}\frac{1}{n} \log E^\go_0\big[e^{\gl H_n}\big] \quad \go-\text{a.s.},
\qquad \gl \in \R.
\end{equation}
\end{proposition}

\begin{figure}[htbp]
\begin{center}
\includegraphics[clip, trim=4.2cm 10cm 3cm 13cm, width=.7\textwidth]{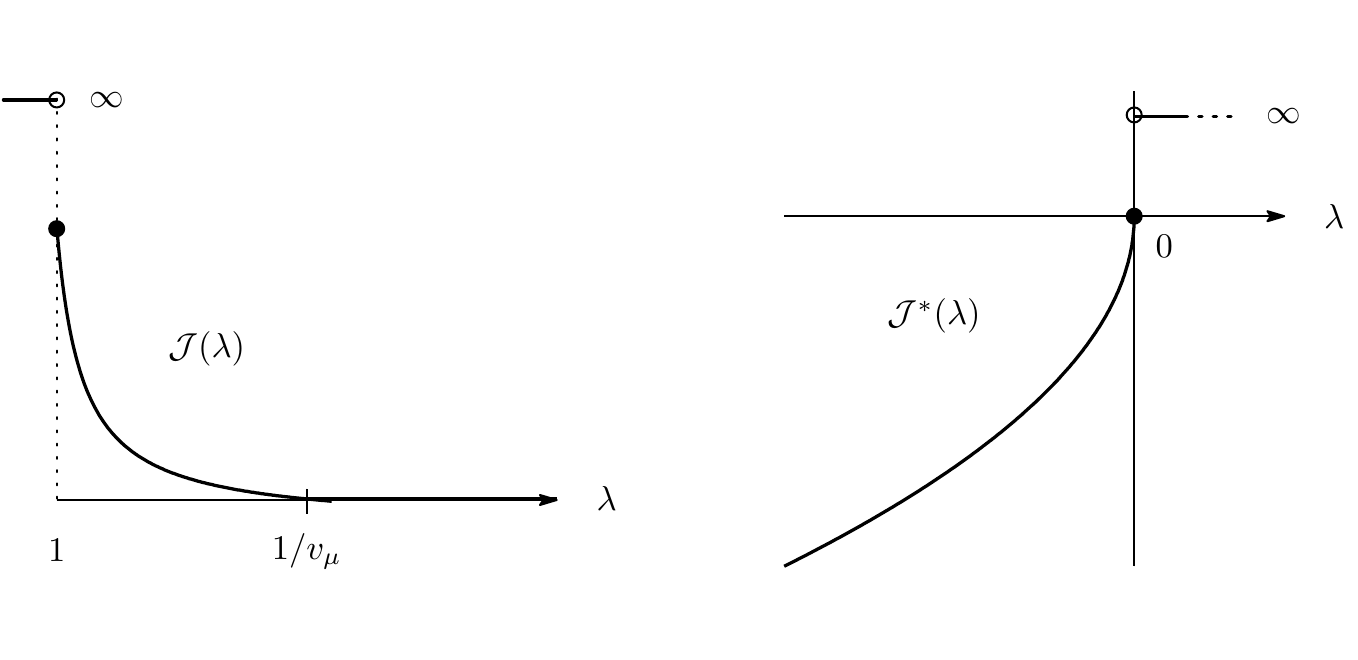}
\caption{\emph{Left:} Graph of $\mc{J}$, the quenched rate function of RWRE hitting 
times in~\eqref{qRF_T}. \emph{Right:} Graph of $\mc{J}^*$, the scaled cumulant generating 
function of RWRE hitting times in~\eqref{qCGF_T}.}
\label{rateJJ*}
\end{center}
\end{figure}

\noindent
For the hitting times to the left, defined by~\eqref{Htimes} with $n\in-\N$, we have the weak rate 
function $\widetilde{\mc{J}}=\widetilde{\mc{J}}_{\mu}$:
\begin{equation}
\widetilde{\mc{J}}(x) = \mc{J}(x) - \langle \log \rho \rangle, \qquad x\in \R.
\end{equation}
Moreover, the following relation between $\mc{J}$ and $\mc{I}$ holds (see ~\cite[Chapter VII]{dH00}):
\begin{equation} 
\label{IJrel}
\mc{I}(x) =
\begin{cases} 
x \mc{J}(1/x), &x \in (0,1],\\
0, &x  = 0,\\
(-x) \mc{J}(1/(-x)), &x \in [-1,0).
\end{cases}
\end{equation}
\begin{figure}[htbp]
\centering
\includegraphics[clip, trim=1.0cm 10cm 1cm 14cm, width=1\textwidth]{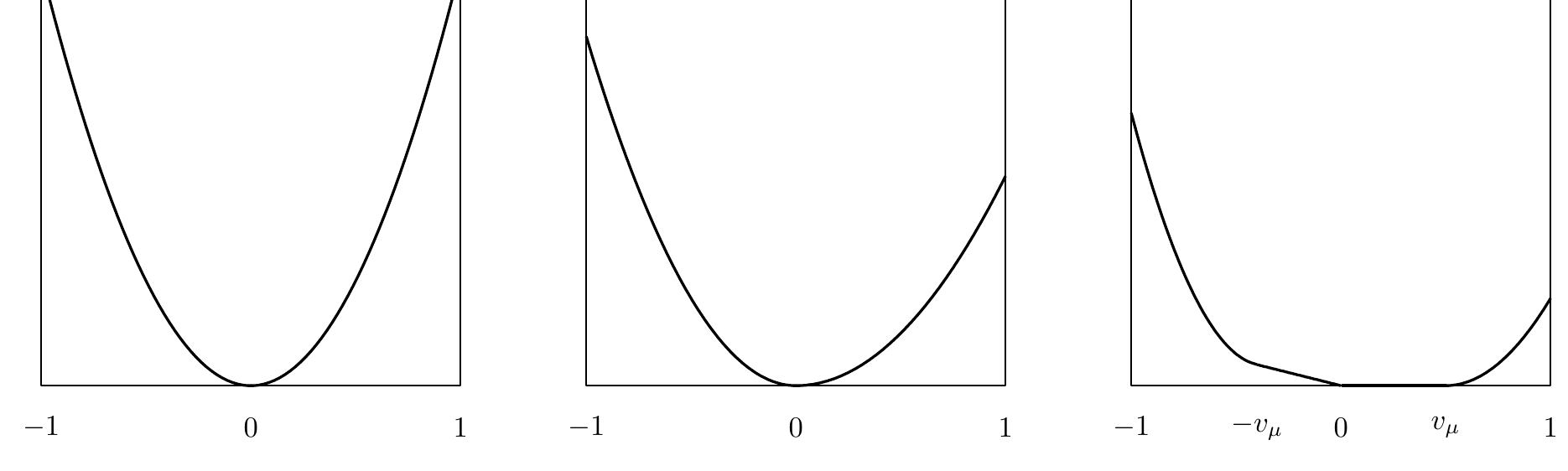}
\caption{Graph of $\mc{I}$, the quenched rate function of RWRE displacements in~\eqref{IJrel}. 
Three cases are shown from left to right: recurrent, transient with zero speed, transient with positive speed. 
}
\label{fig:Irf}
\end{figure}
 
\noindent

The empirical speed of RWRE also satisfies the LDP under the annealed law. 

\begin{proposition}[{\bf Annealed LDP for RWRE displacements}~\cite{CGZ00}] 
\label{annLDP}
Suppose that $\mu$ is basic. Then $(Z_n/n)_{n\in\N}$ under $P^\mu_0$ satisfies the LDP on $\R$
with rate $n$ and with a convex rate function $\mc{I}^{\text{ann}}=\mc{I}^{\text{ann}}_\mu$.
\end{proposition}

\noindent
As shown in~\cite{CGZ00}, the annealed and the quenched rate function are related through the 
following variational principle 
\begin{equation}
\mc{I}^{\text{ann}}(\theta) = \mc{I}^{\text{ann}}_\mu(\theta)
= \inf_{\nu}\big[\mc{I}_\nu(\theta)+|\theta|\,h(\nu \mid \mu)\big],
\end{equation}
where  $\mc{I}_\nu$ is the quenched rate function associated with a random environment that has law $\nu$, 
$h(\nu \mid \mu)$ denotes the relative entropy of $\nu$ with respect to $\mu$, and the infimum runs over the 
set of probability measures on $(0,1)^\Z$ endowed with the weak topology (see~\cite{CGZ00} for more 
details). In particular, $\mc{I}^{\text{ann}}$ is qualitatively similar to  $\mc{I}$ in Fig.~\ref{fig:Irf}, in the 
sense that $\mc{I}^{\text{ann}}$ is strictly decreasing on $[-1,0]$, zero on $[0,v_\mu]$, and strictly 
increasing on $[v_\mu,1]$. The presence of the flat piece $[0,v_\mu]$ in the positive speed case 
makes our analysis more delicate, and we will need the following large deviation bound characterising 
the right decay when zooming in on the flat piece:
\begin{proposition}[{\bf Refined annealed large deviations in the flat piece}~\cite{DPZ96}] 
\label{flatannLDbound}
Suppose that $\mu$ is basic and that $\langle \rho \rangle < 1$. Then there is a unique $s > 1$ 
satisfying $\langle \rho^s \rangle = 1$ such that, for any $\mathcal{O} \subset (0,v_\mu)$ open 
and separated from $v_\mu$,
\begin{equation}
\lim_{n\to\infty} \frac{1}{\log n} \log P_0^\mu\prt{\frac{Z_n}{n} \in \mathcal{O}} = 1 - s.
\end{equation}
\end{proposition}


\subsection{RWCRE: Cooling} 
\label{model}

The cooling random environment is the \emph{space-time} random environment built by partitioning 
$\N_0$, and assigning independently to each piece an environment sampled from $\mu$ in~\eqref{alpha} 
(see Fig.~\ref{fig:CRE}). Formally, let $\tau \colon\, \N_0 \to \N_0$ be a strictly increasing function 
with $\tau(0) = 0$, referred to as the \emph{cooling map}. The cooling map determines a sequence 
of \emph{refreshing} times $\prt{\tau(k)}_{k\in \N_0}$ that we use to construct the dynamic random 
environment.

\begin{definition}[{\bf Cooling Random Environment}] 
{\rm Given a cooling map $\tau$, let $\Omega=(\omega_k)_{k\in\N}$ be an i.i.d.\ sequence of random 
variables with law $\mu$ in~\eqref{alpha}. The \emph{cooling random environment} is built from the pair 
$(\Omega,\tau)$ by assigning the environment $\omega_k$ to the $k$-th interval $I_k$ defined by }
\begin{equation}
\label{interval}
I_k =[\tau(k-1), \tau(k)), \quad \quad k \in \N.
\end{equation}
\end{definition}

In the present paper we consider the \emph{cooling regime}, i.e., we consider $\tau$ such that the 
length of $I_k$ in~\eqref{interval} diverges:
\begin{equation} 
\label{coolingreg}
 T_k = \tau(k) - \tau(k - 1), \qquad \lim_{k\to\infty} T_k = \infty.
\end{equation}
The role of this assumption is clarified in Section~\ref{discussion}.  

\begin{figure}[htbp]
\vspace{0.5cm}
\begin{center} 
\includegraphics[clip, trim=3.0cm 10.5cm 3.2cm 15cm, width=0.9\textwidth]{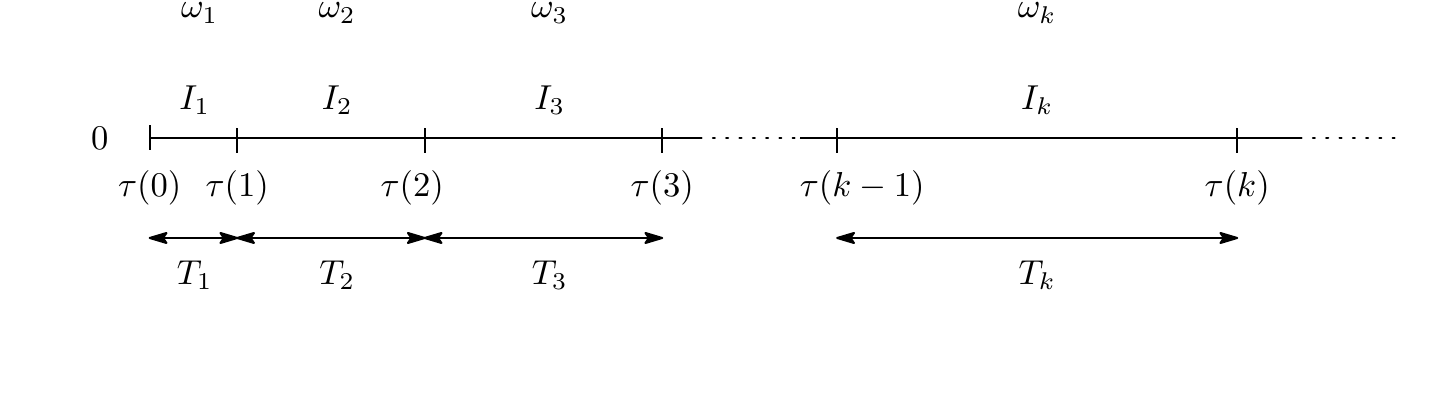}
\vspace{-1.2cm}
\caption{Structure of the cooling random environment $(\Omega,\tau)$.}
\label{fig:CRE}
\end{center}
\end{figure}

\begin{definition}[{\bf RWCRE}] 
{\rm Let $\tau$ be a cooling map and $\Omega$ an environment sequence sampled from $\mu^\N$.
We call \emph{Random Walk in Cooling Random Environment}  the Markov chain 
$X = (X_n)_{n\in\N_0}$ with state space $\Z$ and transition probabilities} 
\begin{equation} 
\label{Ker_Om}
P^{\Omega,\tau}(X_{n+1} = x + e \mid X_n = x) 
= \left\{
\begin{array}{ll}
\omega_{\ell(n)}(x), &e = 1,\\
1 - \omega_{\ell(n)}(x), &e = - 1,
\end{array}
\right. 
\qquad n \in \N_0,
\end{equation}
{\rm where} 
\begin{equation}
\label{ell}
\ell(n) = \inf\{k\colon\, \tau(k) > n\}
\end{equation}
{\rm is the index of the interval $n$ belongs to. Similarly to Definition~\ref{RWREdef}, we denote by} 
\be{RWCREmeas}
P_x^{\Omega,\tau}(\cdot) \quad\text{ and }\quad  P_{x}^{\mu,\tau}(\cdot) 
= \int P_x^{\Omega,\tau}(\cdot)\,\mu^\N(\dd\Omega),
\ee
{\rm the corresponding \emph{quenched} and \emph{annealed} laws, respectively.}
\end{definition} 

\noindent
In words, RWCRE moves according to a given environment sampled from $\mu$, until the next 
refreshing time $\tau(k)$, when a new environment is sampled from $\mu$. Equivalently, the 
random walk trajectory is independent across the intervals, and during each interval $I_k$ 
moves like a RWRE in the environment $\omega_k$. In view of assumption~\eqref{coolingreg}, 
the environment is resampled along a diverging sequence of time increments. Our goal is to 
understand in what way this makes RWCRE behave similarly as RWRE (see Section~\ref{discussion} 
below).       

The position $X_n$ of RWCRE admits the following key decomposition into pieces of RWRE.
Define the \emph{refreshed increments} and the \emph{boundary increment} as
\begin{equation} 
\label{space_incr}
Y_k = X_{\tau(k)} - X_{\tau(k-1)}, \quad k \in \N, \qquad \bar{Y}^n = X_n - X_{\tau(\ell(n)-1)},
\end{equation}
and the \emph{running time at the boundary} as 
\begin{equation} 
\label{time_incr}
\qquad \bar{T}^n = n - \tau(\ell(n)-1).
\end{equation}
Note that, by~\eqref{coolingreg},
\begin{equation} 
\label{time_incr*}
\sum_{k=1}^{\ell(n)-1} T_k + \bar{T}^n = n.
\end{equation}
By construction, we can write $X_n$ as the sum
\begin{equation}
\label{X_dec}
X_n = \sum_{k=1}^{\ell(n)-1} Y_k + \bar Y^n, \quad  n\in\N_0.
\end{equation}
This decomposition shows that, in order to analyse $X$, we must analyse the vector 
\begin{equation}
  (Y_1,\cdots,Y_{\ell(n)-1},\bar{Y}^n)
\end{equation}
consisting of independent components, each distributed as an increment of $Z$ (defined in 
Section~\ref{RWRE}) over a given time length determined by $\tau$ and $n$. Fig.~\ref{fig:RWCRE} 
illustrates this piece-wise decomposition of $X_n$. More precisely, for any measurable function 
$f\colon\,\Z\to\R$, any $\Omega$ sampled from $\mu^\N$ and any $\tau$,
\begin{equation} 
\label{equiv_distr}
E^{\Omega,\tau}_0\crt{f(Y_k)} = E^{\omega_k}_0\crt{f(Z_{T_k})}, 
\qquad E^{\Omega,\tau}_0\crt{f(\bar{Y}^n)} = E^{\omega_{\ell(n)}}_0\crt{f(Z_{\bar{T}^n})}.
\end{equation}
\begin{figure}[htbp]
\begin{center}
\includegraphics[clip, trim=1.8cm 10cm 1.2cm 14cm, width=.95\textwidth]{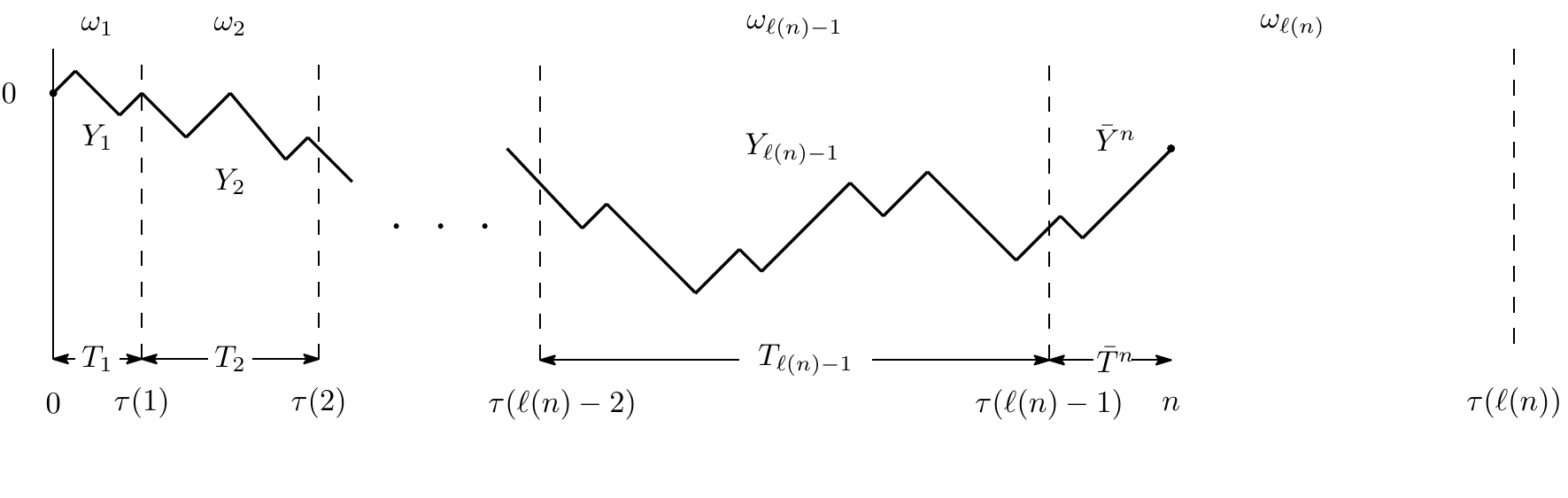}
\vspace{-.5cm}
\caption{The decomposition of RWCRE in pieces of RWRE as presented in~\eqref{X_dec}.}
\label{fig:RWCRE}
\end{center}
\end{figure}

\subsection{Main results} 
\label{results}

We can now state our main results for the asymptotic behaviour of RWCRE.

\begin{theorem}[{\bf SLLN for RWCRE displacements}]
\label{SLLN}
Suppose that $\mu$ is basic and that $\tau$ satisfies~\eqref{coolingreg}. Then, for $\mu^{\bb{N}}$- a.e.\ 
$\Omega$,
\begin{equation}
\lim_{n\to\infty} \frac{X_n}{n} = v_\mu \qquad P^{\Omega,\tau}_0-\text{a.s.}
\end{equation}
with $v_\mu$ as in~\eqref{speed}.
\end{theorem}

\begin{theorem}[{\bf Quenched LDP for RWCRE displacements}]
\label{thm:LDP}
Suppose that $\mu$ is basic and that $\tau$ satisfies~\eqref{coolingreg}. Then, for $\mu^\N$-a.e.\ $\Omega$, 
$(X_n/n)_{n\in\N}$ under $P^{\Omega,\tau}_0$ satisfies the LDP on $\R$ with rate $n$ and with the same
rate function $\mc{I}=\mc{I}_\mu$ as in Proposition~\ref{staticLDPque}.
\end{theorem}

Both theorems will be discussed in Section~\ref{discussion} and will be proved in Section~\ref{SLLDP}.
Their derivation will be based on the following general limit property tailored to RWCRE.

\begin{theorem}[{\bf Cooling Ergodic Theorem}]
\label{CEThm}
Let $(\psi_n^{(k)})_{n,k\in\N}$ be an array of real-valued random variables with law $\PP$ such that
the following assumptions hold:
\begin{itemize}
\item[$(A1)$] 
For all $k,k'\in\N$ with  $k\neq k'$, $(\psi_n^{(k)})_{n\in\N}$ and $(\psi_n^{(k')})_{n\in\N}$ are independent.
\item[$(A2)$] 
For all $k\in\N$, $\sup_{n\in\N} |\psi_{n+1}^{(k)}-\psi_n^{(k)}| \leq C$ for some $C>0$.
\item[$(A3)$] 
There exist $L\in\R$, $\delta>0$,  such that for all $\gep>0$, there is a $C' = C'(\gep)$ for which
\begin{equation}
\label{eq:A3}
\PP\left(\abs{\frac{\psi_n^{(k)}}{n}-L}>\gep\right)<\frac{C'}{n^\delta}.
\end{equation}
\end{itemize}
Then, for any cooling map $\tau$ satisfying~\eqref{coolingreg}, 
\begin{align}
\label{CoolingSum}
\lim_{n\to\infty}\frac{1}{n}\left(\sum_{k =1}^{\ell(n)-1}\psi_{T_k}^{(k)}
+\psi_{\bar{T}^n}^{\left(\ell(n)\right)}\right) = L \qquad \PP-a.s. 
\end{align} 
with $\ell(n)$ as in~\eqref{ell}, $T_k$ as in~\eqref{coolingreg} and $\bar{T}^n$ as in~\eqref{time_incr}.
\end{theorem}

Theorem~\ref{CEThm} is useful for controlling limits of sums of the form appearing in~\eqref{CoolingSum}. 
Its proof is presented in Section~\ref{CETp} and is based on moments bounds and concentration estimates, 
applied to a further decomposition into what we call refreshed, boundary and deterministic terms, respectively. 
Theorem~\ref{CEThm} is a key ingredient in our paper. 

To check Assumption $(A3)$ is a challenge. In itself, $(A3)$ is only a mild decay requirement, but it 
forces us to derive concentration inequalities for RWRE, which is a non-trivial task. For the SLLN
in Theorem~\ref{SLLN}, the required concentration inequalities are already at our disposal, since 
they are encoded in the annealed large deviation bound recalled in Proposition~\ref{flatannLDbound}. However, 
for Theorem~\ref{thm:LDP} a concentration inequality for the cumulants of the displacements of
RWRE is needed which, to the best of our knowledge, is not available in the literature. Therefore
we state such a result, which is of interest in itself.
 
\begin{theorem}[{\bf Concentration of cumulants for RWRE displacements}]
\label{P:MB}
Suppose that $\mu$ is basic. Then, for any $\lambda\in\R$, $\delta \in (0,1)$ and $\gep>0$ 
there are $C,c \in (0,\infty)$ (depending on $\mu,\lambda,\delta,\gep$) such that
\begin{equation}
\label{e:MB}
\mu\prt{\omega \colon \abs{\frac{1}{n} \log E^{\omega}_0\crt{e^{\lambda Z_n}}
- \mc{I}^*(\lambda)}> \gep} 
\leq C \, e^{- c n^{1-\delta}} \qquad \forall\,n \in\N,
\end{equation}
where $\mc{I}^*$ is the Legendre transform of the rate function $\mc{I}$ in 
Proposition~\ref{staticLDPque}, i.e.,
\begin{equation}
\mc{I}^*(\lambda) = \sup_{x \in \R} \big[\gl x - \mc{I}(x)\big], \qquad \gl \in \R.
\end{equation}
\end{theorem}

\noindent
The proof of Theorem~\ref{P:MB} is given in Section~\ref{Cumu}. The idea is to prove an analogous 
concentration inequality for the hitting times, and then transfer this to displacements via successive 
approximations. As is usual in the context of RWRE, hitting times are easier to handle and their 
concentration will follow from an adaptation of an argument presented in the proof of Lemma 3.4.10 in
Zeitouni~\cite{Z04}.


\subsection{Discussion} 
\label{discussion}

\paragraph{No-cooling: bounded time increments.}
The regime where assumption~\eqref{coolingreg} does not hold and the increments $T_k$ in~\eqref{coolingreg} 
are of order one has been investigated in~\cite{AdH17}. Due to the fast resampling, no trapping effects enter 
the game and full homogenisation takes place. In fact, the decomposition in~\eqref{X_dec} gives us a sum 
of almost i.i.d.\ random variables and the resulting behaviour is as if $X$ were a homogeneous Markov chain:     
\cite[Theorem 1.4]{AdH17} shows a corresponding classical SLLN and classical CLT under the annealed law.

\paragraph{Weak and strong law of large numbers.}
Theorem~\ref{SLLN} states that, as soon as the cooling is effective, i.e., assumption~\eqref{coolingreg} is in 
force, the asymptotic speed exists a.s., is deterministic and is equal to the one for RWRE. The same statement 
has been derived in weak form in~\cite[Theorem 1.5]{AdH17}. The strong form presented here requires a 
much more involved proof, based on RWRE concentration inequalities. In fact, Theorem~\ref{SLLN} is far
from trivial because the cooling map allows for fluctuations that could in principle hamper the almost sure 
convergence. As the proof reveals, the fact that this is  not the case comes from a non-trivial averaging 
due to the cooling resampling mechanism. Roughly speaking, the slower the cooling, the stronger are the 
fluctuations of the constituent pieces in the sum in~\eqref{X_dec}, but these fluctuations average out, 
as will be shown with the help of the moments bounds and the concentration estimates mentioned 
earlier.

\paragraph{Large deviations and fluctuations.}
As soon as the increments between the resampling times diverge, the rate function in the LDP for RWCRE 
in Theorem~\ref{thm:LDP} is the same as for RWRE. In words, the cost to deviate from the typical speed 
is determined by the trapping in a fixed environment and the resampling has \emph{no further homogenising 
effect}. This is true when we look on an exponential scale, but we may expect RWRE and RWCRE to show 
different large deviation behaviour when we zoom in on the flat piece $[0,v_\mu]$ when $v_\mu>0$ (see 
Fig.~\ref{fig:Irf}). Theorem~\ref{thm:LDP} deserves further comments because, when we look at fluctuations, 
RWRE and RWCRE actually give rise to \emph{different scaling limits}. This has been proved for recurrent 
RWRE, which exhibits non-standard fluctuations after scaling by $\log^2n$. Indeed,~\cite[Theorem 1.6]{AdH17} 
shows that, for certain $\tau$'s under the annealed law, RWCRE exhibits Gaussian fluctuations after scaling 
by a factor that \emph{grows faster than} $\log^2 n$ and \emph{depends on the cooling map} $\tau$. Similar 
scenarios, and even the presence of a crossover, have been conjectured to hold for RWCRE in other regimes 
(see Table~\ref{tab:a}). This may all sound paradoxical, because it is folklore to expect that the zeros of the 
rate function in an LDP encode information on the order of the fluctuations. However, the latter is only true 
when the rate function is smooth near its zeros. This is not the case for the rate functions in Fig.~\ref{fig:Irf}, 
and so the paradox is explained. 

\paragraph{Relaxing the i.i.d.\ assumption on $\mu$.}
It is worth mentioning that the i.i.d.\ assumption on $\mu$ made in \eqref{alpha} can in principle be relaxed to the 
assumption that $\mu$ is stationary and ergodic with respect to translations. The reader is invited to 
check that all the steps in the proofs below work in this more general setting. On the other hand, we 
will make use of certain known properties of RWRE some of which require further technical assumptions to guarantee
local product structure (see e.g.~\cite[Theorem 2.4.3, p.\ 236]{Z04} for the extension of Proposition~\ref{flatannLDbound}).

\paragraph{Comparison and open problems.}
We conclude by summarising our present understanding of RWCRE based on the results derived here 
and in~\cite{AdH17}. Let us stress again that the RWCRE model can be seen as a model that interpolates 
between the classical static model (i.e., $\tau(1) = \infty$) and the model with i.i.d.\ resamplings every unit
of time (i.e., $\tau(n) = n$). The latter reduces to a homogeneous nearest-neighbour random walk under the annealed 
measure, but even under the quenched law the independent space-time structure leads to a strong 
homogenising scenario for which e.g.\ a classical CLT holds (see e.g.~Boldrighini, Minlos and 
Pellegrinotti~\cite{BMP04}). The interesting features therefore appear as we explore different cooling 
regimes, which allow for a competition between the effect of traps in the static environment and the 
effect of homogenisation coming from the resampling. Table~\ref{tab:a} gives a qualitative comparison 
for RWRE, RWCRE and standard homogeneous nearest-neighbour random walk, abbreviated as RW. 
In view of the discussion above, the no-cooling regime $\tau(n)\sim n$ is in the same ``universality 
class'' as homogeneous random walk, which is why it is put in the same column as RW.

\begin{table}[!htbp]
\label{table}
\vspace{.2cm}
\small
\centering
\begin{tabular}{|c|c|c|c|}
\cline{1-4}
\cline{2-4} Model & RW $\simeq$ No-Cooling &RWCRE &RWRE\\ \hline
\cline{2-4} Medium
&Homogeneous &Cooling &Static\\ \hline \hline
Recurrence & local drift $=0$ &{\bf global, depending on $(\tau,\mu)$?}  
& $\langle\log\rho\rangle=0$, global \\ \hline 
Speed &local drift &$v_{\mu}$ (non-local) &$v_{\mu}$ (non-local)\\ \hline 
LDP rate $n$ &Cram\'er-analytic rate fn& non-analytic rate fn $\mathcal{I}$ 
& non-analytic rate fn $\mathcal{I}$ \\\hline 
Fluctuations & \multirow{5}{2em}{CLT} &$\log\tau(n)\leq cn:$ Gaussian & Kesten-Sinai \\
$\langle\log\rho\rangle=0$ &  &{\bf $\log\tau(n)\geq cn$: Kesten-Sinai?} 
&scale: $\log^2 n$\\
\cline{1-1}\cline{3-4}
Fluctuations &   &
& Kesten-Kozlov-Spitzer\\
\multirow{2}{4.8em}{$\langle\log\rho\rangle<0$} & & {\bf ???}  &$s<2$ stable law\\
&&& $s>2$ CLT\\ \hline
\end{tabular}
\caption{Comparison among standard RW, RWCRE and RWRE.  Marked in boldface are 
what we consider challenging open problems.}
\label{tab:a}
\normalsize
\end{table}

\noindent
Let us comments on the most relevant items in Table~\ref{tab:a}.

\begin{itemize}

\item\emph{Recurrence vs Transience:}
While for a homogenous RW we know that it is recurrent if and only if the corresponding \emph{local} 
drift is zero, for RWRE the recurrence criterion is encoded in the condition $\langle\log\rho\rangle = 0$ 
(recall Proposition~\ref{prop:LLN}). In particular, it can happen that the local drift is non-zero, but still
the above condition holds and the random walk is recurrent. In fact, a random walk in a non-homogeneous 
environment builds up non-negligible correlations over time, and its long-time behaviour is a truly 
\emph{global} feature. For RWCRE we expect some subtle surprises related to the fluctuations of 
the corresponding RWRE. In particular, we expect a non-local criterion as for RWRE, but controlled 
by a delicate interplay between the environment law $\mu$ and the cooling map $\tau$. 

\item\emph{Asymptotic Speed:}
As for the recurrence criterion, the asymptotic speed of a homogeneous RW is given by its local drift, 
while for RWRE it is influenced by the presence of the traps. Theorem~\ref{SLLN} shows that for any 
cooling map subject to~\eqref{coolingreg}, RWCRE has the same speed as RWRE. In other words, on 
scale $1/n$ the long-time behaviours in the two models are equivalent and emerge as global features.

\item\emph{Large Deviations:}
Concerning large deviations of order $n$, for homogeneous RW displacements Cram\'er's theorem tells 
us that their probabilities decay exponentially fast and are determined by a smooth rate function (see 
e.g.~\cite[Chapter I]{dH00}). On the other hand, as we saw in Section~\ref{RWRE}, large deviations 
for RWRE are drastically different, both under the quenched and the annealed measure. In particular, 
both rate functions are non-analytic, not strictly convex when $\langle\log\rho\rangle \neq 0$, and 
contain an interval of zeros when $v_\mu>0$. As previously discussed, Theorem~\ref{thm:LDP} says 
that RWCRE satisfies the same LDP at rate $n$ under the quenched law. Still, we may expect differences 
between quenched large deviations for RWRE and RWCRE when zooming in on the flat piece, i.e., 
when considering decays that are slower than the rate $n$ encoded in the LDP. This constitutes
yet another interesting open problem. Let us further note that we have not looked at the annealed 
LDP for RWCRE. We expect no surprises, namely, we believe that the annealed rate function for
RWCRE is the same as the one for RWRE in Proposition~\ref{annLDP}. In fact, the proof presented 
in Section~\ref{LDPproof} could be easily adapted (and even significantly simplified) if we had existence 
and convexity in the annealed setting. In the quenched setting, existence and convexity are guaranteed 
by a general result derived in~Campos \emph{et.\ al}~\cite{CDRRS13}. 

\item\emph{Fluctuations:}
We conclude with what we consider to be the most challenging open problem, namely, to characterise the
fluctuations for RWCRE, both under the quenched and the annealed measure. Some noteworthy 
results in this direction were derived in~\cite{AdH17}, where an annealed CLT for the no-cooling 
regime was shown~\cite[Theorem 1.4]{AdH17} and, for $\langle\log\rho\rangle=0$ and $\tau$ growing 
either polynomially or exponentially, the annealed centered RWCRE displacement was shown to converge 
to a Gaussian law after an appropriate scaling that depends on $\tau$~\cite[Theorem 1.6]{AdH17}. We 
expect that for sufficiently fast cooling a crossover occurs when $\langle\log\rho\rangle=0$, namely, we 
expect to see the Kesten~\cite{K86} limit law just as in the static case. What happens when $\langle\log
\rho\rangle\neq 0$ seems to be even more intricate and remains fully unexplored. In this case for 
RWRE, Kesten, Kozlov and Spitzer~\cite{KKS75} proved that annealed fluctuations can be Gaussian 
or can be characterised by proper stable law distributions, depending on the value of the root $s \in 
(1,\infty)$ defined in Proposition~\ref{flatannLDbound}. It is reasonable to expect a rich pallet of 
behaviour depending on the interplay between $\tau$ and $s$. The quenched fluctuations seem 
even more difficult to analyse in view of the corresponding more delicate results for RWRE 
(see Zeitouni~\cite{Z04}).
\end{itemize}


\section{Cooling ergodic theorem} 
\label{CETp}

In this section we prove Theorem~\ref{CEThm}.  

\begin{proof}
We represent the sum in~\eqref{CoolingSum} as the convex combination
\begin{equation} 
\label{convex_combi}
\begin{aligned}
&\frac{1}{n} \left( \sum_{k=1}^{\ell(n)-1} \psi_{T_k}^{(k)} + \psi_{\bar{T}^n}^{\left( \ell(n) \right)} \right) 
= \sum_{k=1}^{\ell(n)-1} \frac{T_k}{n} \, \frac{\psi_{T_k}^{(k)}}{T_k} 
+ \frac{\bar{T}^n}{n} \, \frac{\psi_{\bar{T}^n}^{\left(\ell(n)\right)}}{\bar{T}^n},
\end{aligned}
\end{equation}
and use the abbreviations
\begin{equation} 
\label{gamma_phi}
\gamma_{k,n} = \frac{T_k}{n} \Ind{\{k \leq \ell(n)-1\}}, \quad \bar\gamma^n = \frac{\bar{T}^n}{n}.
\end{equation} 
To prove~\eqref{CoolingSum}, we subtract $L$ from~\eqref{convex_combi} and center each 
term in~\eqref{convex_combi}:
\begin{equation} 
\label{RBD}
\begin{aligned}
&\sum_{k\in\N} \gamma_{k,n} \frac{\psi^{(k)}_{T_k}}{T_k} 
+ \bar{\gamma}^n \frac{\psi^{\left( \ell(n) \right)}_{\bar{T}^n}}{\bar{T}^n} - L\\
&= \underbrace{\Bigg(\sum_{k \in \N} \gamma_{k,n} \mc{C}_k\Bigg)} + \underbrace{{\color{white} 
\Bigg(} \bar{\gamma}^n \bar{\mc{C}}^n{\color{white}\Bigg)}}  
+ \underbrace{\Bigg(\sum_{k \in \N} \gamma_{k,n} (L_k - L) + \bar{\gamma}^n (\bar{L}^n - L)\Bigg)}\\
&= \quad\,\,\quad R_n \;\;\;\;\quad + \;\;\quad B_n \;\;\;\; + \;\;\;\qquad\qquad\qquad D_n
\end{aligned}
\end{equation}
where 
\begin{equation} 
\label{GeneralCentering}
\mc{C}_k = \frac{\psi^{(k)}_{T_k}}{T_k} - L_k, \quad \bar{\mc{C}}^n 
= \frac{\psi^{\left( \ell(n) \right)}_{\bar{T}^n}}{\bar{T}^n} - \bar{L}^n, \qquad
L_k = \EE\crt{\frac{\psi^{(k)}_{T_k}}{T_k}}, \quad 
\bar{L}^n = \EE\crt{\frac{\psi^{\left( \ell(n) \right)}_{\bar{T}^n}}{\bar{T}^n}}.
\end{equation}
The terms $R_n$, $B_n$ and $D_n$ correspond to \emph{refreshed}, \emph{boundary} 
and \emph{deterministic} increments, respectively. In Sections~\ref{Rn},~\ref{Bn} and~\ref{Dn} 
we treat each of these terms separately, and show that they are asymptotically vanishing. 

\subsection{Refreshed term $R_n$ } 
\label{Rn}

In this section we show that 
\begin{equation} 
\label{R_goal}
\limsup_{n\to\infty} \abs{R_n} = 0 \quad \PP-\text{a.s.}
\end{equation}
In view of Assumption $(A3)$, we split the increments of the resampling times according to 
a growth parameter $\gamma > 0$ such that $\gamma \delta > 1$,
\begin{equation} 
\label{R_split}
\begin{aligned}
\sum_{k\in\N} \gamma_{k,n} \mc{C}_k 
&= \underbrace{\sum_{k\in\N} \gamma_{k,n} \mc{C}_k\mathbbm{1}_{\{ T_k \geq k^\gamma \}}}  
+ \underbrace{\sum_{k\in\N} \gamma_{k,n} \mc{C}_k \mathbbm{1}_{\{ T_k < k^\gamma\}}}\\
&= \quad\;\;\qquad R_n^L \;\;\quad\qquad + \quad\quad\quad R_n^S,
\end{aligned}
\end{equation}
which corresponds to the sum of large and small increments, respectively.  The goal is to  
bound both $\limsup_{n\to\infty} |R_n^L|$ and $\limsup_{n\to\infty} |R_n^S|$.

We first treat the sum of large increments $R_n^L$. By Assumption $(A3)$, if $\gamma \delta > 1$,
then
\begin{equation}
\label{R_BC}
\sum_{k\in\N} \PP\prt{|\mc{C}_k| \Ind{\chv{T_k \geq k^{\gamma}}} > \gep} <\infty.
\end{equation}
Applying the Borel-Cantelli lemma, we get
\begin{equation} 
\label{BC_L}
\limsup_{k\to\infty} \mc{C}_k \Ind{\chv{T_k \geq k^\gamma}} \leq \gep \qquad \PP-\text{a.s.}
\end{equation}
Since $\lim_{n\to\infty} \gamma_{k,n} = 0$ for fixed $k \in \N$ and $\sum_{k\in\N} \gamma_{k,n} 
\leq 1$, we obtain
\begin{equation} 
\label{R_L_eps}
\limsup_{n\to\infty} |R_n^L| \leq \gep \qquad \PP-\text{a.s.}
\end{equation}

To deal with the sum of small  increments $R_n^S$, we apply the Markov inequality:
\begin{equation} 
\label{e:S_Markov}
\PP\prt{|R_n^S| > \gep} \leq \frac{1}{\varepsilon^{2N}} \, 
\EE\crt{\prt{\sum_{k=1}^n \gamma^S_{k,n} \mc{C}_k}^{2N}}, 
\qquad \gga^S_{k,n} = \gga_{k,n} \mathbbm{1}_{\{ T_k <  k^\gamma \}}.
\end{equation}
Since the $\mc{C}_k$'s are independent, zero-mean and bounded random variables, when we expand the 
$2N$-th power, all terms with first moment disappear.  Therefore, we can estimate the moments 
as
\begin{equation} 
\label{S_2N}  
\begin{aligned}
&\EE\crt{\prt{\sum_{k=1}^n \gamma^S_{k,n} \mc{C}_k}^{2N}}\\ 
&= \sum_{m=1}^{2N} \underset{p_1,\cdots,p_m\geq 1}{\sum_{p_1+\cdots+p_m=2N}} {2N\choose p_1 \cdots p_m} 
\sum_{k_1>\ldots>k_m} \EE\crt{\prt{\gamma^S_{k_1,n} \mc{C}_{k_1}}^{p_1} 
\times \cdots \times \prt{\gamma^S_{k_m,n} \mc{C}_{k_m}}^{p_m}}\\
&\leq \sum_{m=1}^{N} \underset{p_1,\cdots,p_m\geq 2}{\sum_{p_1+\cdots+p_m=2N}} {2N\choose p_1 \cdots p_m} 
\sum_{k_1>\ldots>k_m} \prt{C \, \frac{k_1^\gamma}{n}}^{p_1} 
\times \cdots \times \prt{C \, \frac{k_m^\gamma}{n}}^{p_m}\\
&\leq \sum_{m=1}^{N} {2 N - m - 1 \choose m - 1} \, n^m \, C^{2N} \, n^{2 N (\gamma - 1)} \leq c_N n^{N(2 \gamma -1)},
\end{aligned}
\end{equation}
where in the third line we use independence (Assumption (A1)) and bound each $\gamma^S_{k_i,n} 
\mc{C}_{k_i}$ by $C \, k_i^\gamma \, n^{-1}$ (Assumption (A2)). The right-hand side of~\eqref{S_2N} 
is summable in $n$ as long as $\gamma < \tfrac12$, because we can choose $N$ arbitrarily large. 
This suggests that we need to further separate the argument.

\paragraph{Case $\delta>2$: single split.} 

If $\delta > 2$, then we can pick $\gamma < \tfrac12$.  From~\eqref{e:S_Markov} and~\eqref{S_2N}, 
we obtain
\begin{equation} 
\label{e:BCS}
\sum_{n\in\N} \PP\prt{|R_n^S| > \gep} < \infty.
\end{equation}
Hence, by the Borel-Cantelli lemma,
\begin{equation} 
\label{R_S_eps}
\limsup_{n\to\infty} |R_n^S| \leq \gep \qquad \PP-\text{a.s.}
\end{equation}
Combining~\eqref{R_split},~\eqref{R_L_eps} and~\eqref{R_S_eps}, we get that for $\delta > 2$,
\begin{equation} 
\label{R_0_bound}
\limsup_{n\to\infty} |R_n| \leq 2\gep \qquad \PP-\text{a.s.}
\end{equation}

\paragraph{Case $\delta <2$:  multi-layer split.} 
\label{Rs3}

If $\delta < 2$, then we must pick $\gamma > \tfrac12$ to satisfy $\gamma \delta > 1$.  Here we 
can no longer use the previous argument to obtain~\eqref{e:BCS}. To overcome this difficulty, we 
implement a multi-layer scheme distinguishing small and large increments according to a growth 
parameter. Take $M \in \N$ such that $\frac{M}{3} \delta > 1$.  Similarly to~\eqref{R_split}, define 
the first split: 
\begin{equation} 
\label{R_+_1split}
\begin{aligned}
R_n = \sum_{k\in\N} \gamma_{k,n} \mc{C}_k 
&= \underbrace{\sum_{k\in\N} \gamma^{1,L}_{k,n} \mc{C}_k} 
+ \underbrace{\sum_{k\in\N} \gamma^{1,S}_{k,n} \mc{C}_k}\\
&= \;\;\quad R_n^{1,L} \quad + \;\;\quad R_n^{1,S},
\end{aligned}
\end{equation}
where
\begin{equation} 
\label{gamma_SL_1}
\gamma^{1,S}_{k,n} = \gamma_{k,n} \Ind{\chv{T_k < k^{1/3}}}, \qquad 
\gamma^{1,L}_{k,n} = \gamma_{k,n} \Ind{\chv{T_k \geq k^{1/3}}}. 
\end{equation}

\medskip\noindent
{\bf 1.}
To estimate $R^{1,S}_n$, as in~\eqref{e:S_Markov} and~\eqref{S_2N} we apply the Markov 
inequality, estimate the moments and obtain
\begin{equation} 
\label{R_P_1S_bound}
\PP\prt{|R_n^{1,S}| > \gep} \leq c_N n^{-N/3}.
\end{equation}
Since we can choose $N > 3$ in~\eqref{R_P_1S_bound}, we conclude that $\PP(|R_n^{1,S}| > \gep)$ 
is summable in $n$ and therefore, by the Borel-Cantelli lemma,
\begin{equation} 
\label{R_+_1S_bound}
\limsup_{n\to\infty} |R_n^{1,S}| \leq \gep \qquad \PP-\text{a.s.}
\end{equation}

\begin{center}
\begin{table}[htbp]
\begin{tabular}{|ccl|c|c|}
\hline
\multicolumn{3}{|c|}{ Multi-layer Split}  &  Moments  & Concentration\\ \hline
$R_n=$  &  $\underline{R_n^{1,L}} +$  &  $R_n^{1,S}$  & $\limsup_n R_n^{1,S}\leq \gep$  &\\
&  $\downarrow$ \hspace{.2cm}  &    &    &\\
&  $R_n^{1,L}=$  &  $\underline{R_n^{2,L}} + R_n^{2,S}$  &  $\limsup_n R_n^{2,S} \leq \gep$  &\\
&    &  \quad $\vdots$  &  $\vdots$ \hspace{2cm}  &\\
&    &  $R_n^{M-1,L} =$ $\boxed{R_n^{M,L}} + R_n^{M,S}$  &  $\limsup_n R_n^{M,S} \leq \gep$  
&  $\limsup_n R_n^{M,L} \leq \gep$\\
\hline
\end{tabular}
\caption{Splitting scheme.}
\label{tab:split}
\end{table}
\end{center}

\medskip\noindent
{\bf 2.} 
To estimate $R_n^{1,L}$, the idea is to it decompose iteratively, as we did with $R_n$ in~\eqref{R_+_1split}, 
and control the small increments with moment bounds until we can apply concentration estimates. The 
resulting scheme is summarised in Table~\ref{tab:split}.

\medskip\noindent
{\bf 2a.}
To build the second split, we relabel the terms in $R_n^{1,L}$, i.e., we choose an ordered subsequence 
$(k^1_j)_{j\in\N}$ such that
\begin{equation} 
\label{subseq_1}
\{ k^1_1,k^1_2,\ldots \} = \{ j \in \N \colon T_j \geq j^{1/3} \}. 
\end{equation}
Denoting by $J(1;n)$ the cardinality of $\{k^1_j\colon \tau(k_j^1) \leq n\}$, we define the second split:
\begin{equation} 
\label{R_+_2split}
\begin{aligned}
R_n^{1,L} = \sum_{j=1}^{J(1;n)} \gamma_{k^1_j,n} \mc{C}_{k^1_j} 
&= \underbrace{\sum_{j=1}^{J(1;n)} \gamma_{k^1_j,n}^{2,L} \mc{C}_{k^1_j}} 
+ \underbrace{\sum_{j=1}^{J(1;n)} \gamma_{k^1_j,n}^{2,S} \mc{C}_{k^1_j}}\\
&= \;\;\;\;\quad R_n^{2,L} \;\;\;\quad + \;\;\;\;\quad R_n^{2,S},
\end{aligned}
\end{equation}
where
\begin{equation} 
\label{gamma_SL_2}
\gamma^{2,S}_{k^1_j,n} = \gamma_{k^1_j,n} \Ind{\chv{T_{k^1_j} < j^{2/3}}}, 
\qquad \gamma^{2,L}_{k^1_j,n} = \gamma_{k^1_j,n} \Ind{\chv{T_{k^1_j} \geq j^{2/3}}}. 
\end{equation}
Next, we abbreviate $n(1;J) = \inf\chv{n \colon J(1;n) = J}$.  Then, since
\begin{equation} 
\label{n_J}
\limsup_{n\to\infty} |R_n^{2,S}| = \limsup_{J\to\infty} |R_{n(1;J)}^{2,S}|,
\end{equation}
it suffices to show that $\limsup_{J\to\infty}  |R_{n(1;J)}^{2,S}| \leq \gep$ $\PP$-a.s. Note that, 
since $T_{k^1_j} \geq j^{1/3}$, we have a lower bound on $n(1;J)$:
\begin{equation} 
\label{n1J}
n(1;J) \geq \sum_{j=1}^{J} T_{k^1_j}  \geq \sum_{j=1}^{J} j^{1/3} \geq c \, J^{4/3},
\end{equation}
which yields
\begin{equation} 
\label{gam_est}
\gamma^{2,S}_{k^1_j,n(1;J)} \leq \frac{j^{2/3}}{c J^{4/3}}\leq \frac{1}{c J^{2/3}}.
\end{equation}
Similarly to the first split, we apply the Markov inequality
\begin{equation} 
\label{2S_Markov}
\PP\prt{\abs{R_{n(1;J)}^{2,S}} > \gep} \leq \frac{1}{\varepsilon^{2N}} \, 
\EE\crt{\prt{\sum_{j=1}^{J} \gamma^{2,S}_{k^1_j,n} \mc{C}_{k^1_j}}^{2N}},
\end{equation}
and estimate moments
\begin{equation} 
\label{2S_2N_est}
\begin{aligned}
&\EE\crt{\prt{\sum_{j=1}^{J} \gamma^{2,S}_{k^1_j,n(1;J)} \mc{C}_{k^1_j}}^{2N}}
\leq c_N \, J^{- N/3}.
\end{aligned}
\end{equation}
Once we choose $N > 3$, this becomes summable in $J$. Therefore, by the Borel-Cantelli lemma 
and~\eqref{n_J}, we obtain
\begin{equation}
\label{R_+_2S_bound}
\limsup_{n\to\infty} |R_{n}^{2,S}| \leq \gep \quad \PP-\text{a.s.}
\end{equation}

\medskip\noindent
{\bf 2b.}
We continue the induction steps. For any $i < M$, after bounding $\limsup_{n\to\infty} |R_n^{i,S}|$, 
we relabel the terms in $R_n^{i,L}$ and define
\begin{equation} 
\label{subseq_i}
\{ k^{i}_1,k^{i}_2,\ldots \} = \{ j \in \N \colon T_{k^{i-1}_j} \geq j^{i/3} \}. 
\end{equation}
Denoting by $J(i;n)$ the cardinality of $\{k^{i}_j\colon \tau(k^i_j) \leq n\}$, we define the $(i+1)$-st split:
\begin{equation} 
\label{R_+_(i+1)split}
\begin{aligned}
R_n^{i,L} = \sum_{j=1}^{J(i;n)} \gamma_{k^{i}_j,n} \mc{C}_{k^i_j} 
&= \underbrace{\sum_{j=1}^{J(i;n)} \gamma_{k^i_j,n}^{i+1,L} \mc{C}_{k^i_j}} 
+ \underbrace{\sum_{j=1}^{J(i;n)} \gamma_{k^i_j,n}^{i+1,S} \mc{C}_{k^i_j}}\\
&= \;\;\;\;\quad R_n^{i+1,L} \;\quad + \,\;\;\;\;\quad R_n^{i+1,S},
\end{aligned}
\end{equation}
where
\begin{equation} 
\label{gamma_SL_i}
\gamma^{i+1,S}_{k^{i}_j,n} = \gamma_{k^{i}_j,n} \Ind{\chv{T_{k^i_j} < j^{(i+1)/3}}}, 
\qquad \gamma^{i+1,L}_{k^i_j,n} = \gamma_{k^i_j,n} \Ind{\chv{T_{k^i_j} \geq j^{(i+1)/3}}}. 
\end{equation}
Let $n(i;J) = \inf\chv{n \colon  J(i;n) = J}$. Then, by a similar computation as in~\eqref{n1J} and 
\eqref{gam_est}, we have the following bounds:
\begin{equation}
n(i;J) \geq c \, J^{1+i/3}, \qquad \gamma^{i+1,S}_{k^i_j,n(i;J)} \leq \frac{1}{c J^{2/3}}.
\end{equation}
Using the Markov inequality and moments bounds, we obtain
\begin{equation}
\sum_{J\in\N} \PP\prt{\abs{R_{n(i;J)}^{i+1,S}} > \gep} < \infty.
\end{equation}
Therefore we conclude that
\begin{equation}
\label{R_+_(i+1)_bound}
\limsup_{n\to\infty} |R_{n}^{i+1,S}| \leq \gep \quad \PP-\text{a.s.}
\end{equation}

\medskip\noindent
{\bf 2c.}
Once we bound $\limsup_{n\to\infty} |R^{M,S}_n|$, we are left with the term $R_n^{M,L}$.  Since
\begin{equation}
R_{n}^{M,L} = \sum_{j\in\N} \gamma_{k^M_j,n}\Ind{\{T_{k^M_j> j^{M/3}}\}} \mc{C}_{k^M_j}
\end{equation}
and $\frac{M}{3} \delta > 1$, we apply Assumption $(A3)$ to obtain 
\begin{equation}
\sum_{j\in\N} \PP\prt{|\mc{C}_{k^M_j}|\Ind{\{T_{k^M_j> j^{M/3}}\}} > \gep} < \infty.
\end{equation}
Hence, by the Borel-Cantelli lemma,
\begin{equation} 
\label{BC_+_L}
\limsup_{j\to\infty} |\mc{C}_{k^{M-1}_j}|\Ind{\chv{T_{k^{M-1}_j > j^{M/3}}}} 
\leq \gep \qquad \PP-\text{a.s.}
\end{equation}
Since $\lim_{n\to\infty} \gamma_{k,n} = 0$ for fixed $k$ and $\sum_{k\in\N} \gamma_{k,n} \leq 1$, 
we obtain
\begin{equation} 
\label{R_+_M_L_bound}
\limsup_{n\to\infty} |R_n^{M,L}| \leq \gep \qquad \PP-\text{a.s.}
\end{equation}

\medskip\noindent
{\bf 3.}
Combining~\eqref{R_+_M_L_bound} and~\eqref{R_+_(i+1)_bound} for $i \in \chv{0,\ldots,M-1}$, 
we conclude that
\begin{equation} 
\label{R_+_bound}
\limsup_{n\to\infty} |R_n| \leq (M+1) \gep \qquad \PP-\text{a.s.}
\end{equation}
and~\eqref{R_goal} follows since $\gep > 0$ is arbitrary.

\subsection{Boundary term $B_n$} 
\label{Bn}

We next show that 
\begin{equation} 
\label{B_goal}
\limsup_{n\to\infty} \abs{B_n} = 0 \qquad \PP-\text{a.s.}
\end{equation}
Let $V_k = \sup\chv{\bar{\gga}_n |\bar{\mc{C}}^n| \colon n \in I_k}$. Because $\cup_{k\in\N} I_k \supset \N_0$, 
we have $\limsup_{n\to\infty} \bar{\gga}_n |\bar{\mc{C}}^n| = \limsup_{k\to\infty} V_k$.  It therefore 
suffices to show that for arbitrary $\gep > 0$,
\begin{equation} 
\label{Bgoal0}
\limsup_{k\to\infty} V_k \leq \gep \qquad \PP-\text{a.s.}
\end{equation}
If $\bar{\gamma}^n \leq \gep$, then using Assumption $(A2)$ we can bound 
$|B_n| \leq C \, \gep$.  Therefore we only need to consider $\bar{\gamma}^n > \gep$
(see Fig.~\ref{fig:gamma_behav*}), in which case we see that
\begin{equation}
\label{danger}
n > \frac{\tau(k-1)}{1 - \gep} = N_{k,\gep}.
\end{equation}

\begin{figure}[htbp]
\centering  
\includegraphics[clip, trim=3.5cm 11cm 2.5cm 14cm, width=1\textwidth]{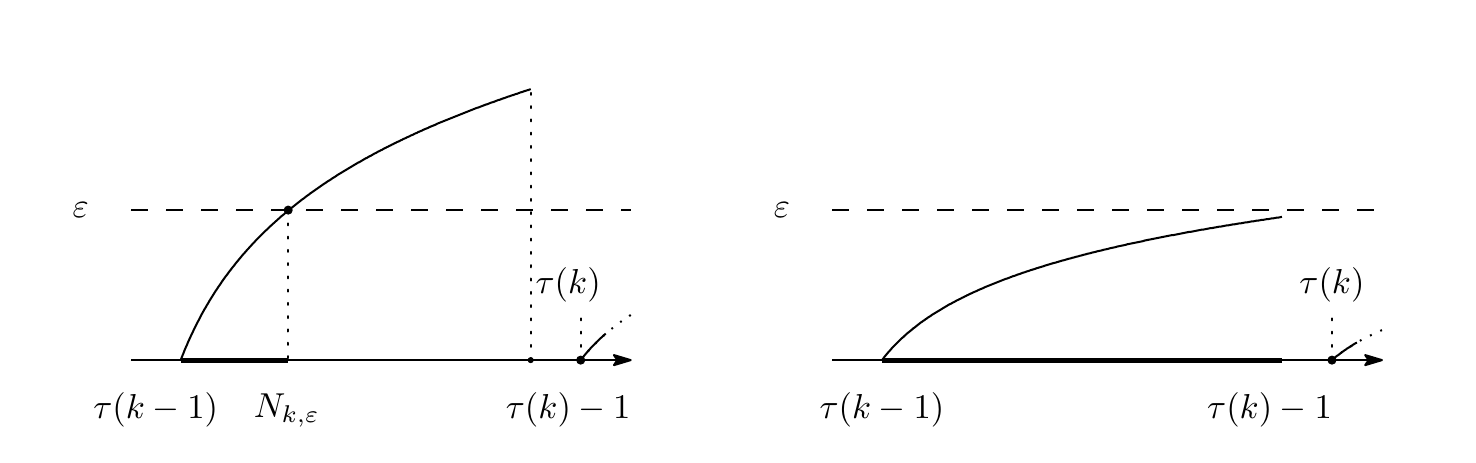}
\caption{\emph{Left:} If $N_{k,\gep} < \tau(k)$, then $\bar{\gga}_n \leq \gep$ for 
$n \in [\tau(k-1),N_{k,\gep})$ and $\bar{\gga}_n > \gep$ for $n \in [N_{k,\gep},\tau(k))$. 
\emph{Right:} If $N_{k,\gep} \geq \tau(k)$, then $\bar{\gga}_n \leq \gep$ for all 
$n \in I_k$.}
\label{fig:gamma_behav*}
\end{figure}

If $\tau(k)\leq N_{k,\gep}$, then the interval $I_k$ can be ignored. Defining
\begin{equation}
\chv{k_1,k_2,\ldots} = \chv{k\in \N\colon\, \tau(k)\geq N_{k,\gep}},
\end{equation}
our task reduces to showing that
\begin{equation} 
\label{Bgoal1}
\limsup_{j\to\infty} V_{k_j} \leq C \gep \qquad \PP-\text{a.s.}
\end{equation}
Note that the subsequence $\prt{\tau(k_j)}_{j\in\N}$ grows at least exponentially fast once
\begin{equation} 
\label{growth}
\tau(k_j) \geq N_{k_j,\gep} > (1 + \gep) \tau(k_j-1) \geq (1 + \gep) \tau(k_{j-1}) 
\geq (1 + \gep)^{j-1} \tau(k_1).
\end{equation}
Since $|B_n| \leq C \gep$ for $n \in [\tau(k_j-1),N_{k_j,\gep})$ and $\bar{\gamma}^n \leq 1$, 
by letting $m = \bar{T}^n = n - \tau(k_j -1)$ and noting that $(1+\gep)\tau(k_j-1)\leq N_{k,\gep}$, 
we obtain 
\begin{equation} 
\label{Vpre}
\begin{aligned}
\PP\prt{V_{k_j} > C \gep} &= \PP\prt{\sup_{N_{k,\gep}\leq n<\tau(k_j)} 
\bar{\gga}_n |\bar{C}^n| > C \gep}\leq \PP\prt{\sup_{N_{k,\gep}\leq n<\tau(k_j)}  
|\bar{C}^n| > C \gep}\\
&\leq \PP\prt{\sup_{\gep\tau(k_j-1)\leq m<T_{k_j}} \left| \frac{\psi^{\left(k_j\right)}_m}{m} 
- \EE\crt{\frac{\psi^{\left(k_j\right)}_m}{m}} \right| > C \gep}.
\end{aligned}
\end{equation}
By the union bound applied to~\eqref{Vpre}, we arrive at
\begin{equation} 
\label{Vpre_est} 
\PP\prt{V_{k_j} > C \gep}\leq \sum_{m=\gep\tau(k_j-1)}^{T_{k_j}} 
\PP\prt{\left| \frac{\psi^{\left(k_j\right)}_m}{m} 
- \EE\crt{\frac{\psi^{\left(k_j\right)}_m}{m}} \right| > C \gep}.
\end{equation}
By Assumption $(A3)$, $\lim_{m\to\infty} \EE[\psi^{\left(k\right)}_m/m] = L$ uniformly in $k$, 
and in particular, from~\eqref{Vpre_est}, we get
\begin{equation} 
\label{B_est}
\begin{aligned}
\PP\prt{V_{k_j} > C \gep} 
&\leq \sum_{m=\gep\tau(k_j-1)}^{T_{k_j}}\frac{\widetilde{C}}{m^{\delta}},
\end{aligned}
\end{equation}
for some $\widetilde{C}>0$ not depending on $k_j$.

\paragraph{Case $\delta > 1$.}

From~\eqref{B_est} we see that
\begin{equation} 
\label{Best}
\begin{aligned}
\PP\prt{V_{k_j} > C \gep} &\leq \frac{C''}{\gep\tau(k_j-1)^{\delta - 1}}.
\end{aligned}
\end{equation}
If $\gd > 1$, then together with~\eqref{growth} this implies that~\eqref{Best} is summable in $j$. 
Hence, by the Borel-Cantelli lemma, we obtain
\begin{equation} 
\label{B_bound}
\limsup_{j\to\infty} V_{k_j} \leq C \gep \qquad  \PP-\text{a.s.}
\end{equation} 

\paragraph{Case $\delta < 1$.}

We need a more refined argument to prove~\eqref{B_goal}. To control the boundary term on the 
interval $I_{k_j}$, we construct a sequence of times $\prt{J_i}_{i\in \N_0}$ such that $J_0 = \gep 
\tau(k_j-1)$ and $J_i = (1 + \gep) J_{i-1}$.  For $m \in (J_i,J_{i+1})$,
\begin{equation}
\abs{\frac{\psi^{(k_j)}_m}{m} - \frac{\psi^{(k_j)}_{J_i}}{J_i}} 
\leq \frac{1}{J_i} \abs{\psi^{(k_j)}_m - \psi^{(k_j)}_{J_i}} \leq C \gep,
\end{equation}
where we use Assumption $(A2)$ in the last inequality. Hence
\begin{equation}
\begin{aligned}
\abs{\frac{\psi^{(k_j)}_m}{m} - \EE\crt{\frac{\psi^{(k_j)}_m}{m}}}
&\leq \abs{\frac{\psi^{(k_j)}_m}{m} - \frac{\psi^{(k_j)}_{J_i}}{J_i}} 
+ \abs{\frac{\psi^{(k_j)}_{J_i}}{J_i} - \EE\crt{\frac{\psi^{(k_j)}_{J_i}}{J_i}}} 
+ \abs{\EE\crt{\frac{\psi^{(k_j)}_m}{m} - \frac{\psi^{(k_j)}_{J_i}}{J_i}}}\\
&\leq \abs{\frac{\psi^{(k_j)}_{J_i}}{J_i} - \EE\crt{\frac{\psi^{(k_j)}_{J_i}}{J_i}}} 
+ 2 C \gep.
\end{aligned}
\end{equation}
Therefore
\begin{equation} 
\label{set_control}
\left\{ \sup_{\gep\tau(k_j -1)\leq m<T_{k_j}} \abs{\frac{\psi^{(k_j)}_m}{m} 
- \EE\crt{\frac{\psi^{(k_j)}_m}{m}}} > 3 C \gep \right\} \subset 
\left\{ \sup_{i\in\N_0} \abs{\frac{\psi^{(k_j)}_{J_i}}{J_i} 
- \EE\crt{\frac{\psi^{(k_j)}_{J_i}}{J_i}}} > C \gep \right\}.
\end{equation}
Hence, arguing as in~\eqref{Vpre} and using~\eqref{set_control}, the union bound and Assumption 
$(A3)$, we can estimate
\begin{equation}
\begin{aligned}
\PP\prt{V_{k_j} > 3 C \gep} &= \PP\prt{\sup_{\gep\tau(k_j-1)\leq m<T_{k_j}} \abs{\frac{\psi^{(k_j)}_m}{m} 
- \EE\crt{\frac{\psi^{(k_j)}_m}{m}}} > 3 C \gep}\\
&\leq \PP\prt{\exists\, i \in \N_0 \colon \abs{\frac{\psi^{(k_j)}_{J_i}}{J_i} 
- \EE\crt{\frac{\psi^{(k_j)}_{J_i}}{J_i}}} > C \gep}\\
&\leq \sum_{i\in\N_0} \frac{C'}{J_i^{\delta}} 
= \sum_{i\in\N_0} \frac{C'}{(1 + \gep)^{i} J_0^{\delta}} 
\leq \frac{C'''}{J_0^\delta} = \frac{C'''}{\prt{\gep \tau(k_j-1)}^\delta},
\end{aligned}
\end{equation}
which is summable in $j$ due to~\eqref{growth}. By the Borel-Cantelli lemma, we conclude that
\begin{equation} 
\label{B_+_bound} 
\limsup_{n\to\infty} |B_n| \leq 3 C \gep \qquad \PP-\text{a.s.}
\end{equation}
Since $\gep > 0$ is arbitrary,~\eqref{B_+_bound} and~\eqref{B_bound} imply~\eqref{B_goal}.

\subsection{Deterministic term $D_n$} 
\label{Dn}

To conclude the proof of Theorem~\ref{CEThm}, it remains to show a.s.\ convergence to zero of 
$D_n$ in~\eqref{RBD}.  We make use of the following statement derived in~\cite[Lemma 3.1]{AdH17}, 
which is a variant of the so-called Toeplitz lemma tailored to RWCRE.

\begin{lemma} 
\label{lem:Toep}
Let $(\gamma_{k,n})_{k,n\in\N}$, $\bar\gamma^n$ be as in~\eqref{gamma_phi} and $\bar T^n$ 
be as in~\eqref{time_incr}. Let $(z_k)_{k\in\N}$ be a real-valued sequence such that $\lim_{k\to\infty} 
z_k = z^*$ for some $ z^*\in\R$.  Then
\begin{equation}
\lim_{n\to \infty} \left(\,\sum_{k \in \N} \gamma_{k,n} z_k + \bar\gamma^n z_{\bar T^n}\right) = z^*.
\end{equation}
\end{lemma}

\noindent 
Recall that $L = \lim_{k\to\infty} L_k$.  By Lemma~\ref{lem:Toep} with $z_k = L_k - L$ and $z^*=0$, 
we conclude that 
\begin{equation}
\label{D_goal}
\limsup_{n\to\infty} D_n = 0 \qquad \PP-a.s. 
\end{equation}

\medskip
Combining~\eqref{R_goal},~\eqref{B_goal} and~\eqref{D_goal}, we get the claim in~\eqref{CoolingSum}. 
\end{proof}


\section{Concentration of cumulants for RWRE displacements} 
\label{Cumu}

The proof of Theorem~\ref{P:MB} will be divided into four steps, organised in 
Sections~\ref{proofTn}--\ref{ss:halfhit}. The basic idea is to derive a concentration inequality for 
the hitting times of RWRE, which are easier to analyse, and then transfer this to a concentration 
inequality for the  displacements of RWRE.

Here is the analogue of Theorem~\ref{P:MB} for the hitting times defined in~\eqref{Htimes}. 

\begin{proposition}[{\bf Concentration of cumulants for RWRE hitting times}] 
\label{L:ZT}
Suppose that $\mu$ is basic. Then, for any $\lambda\in \R$, $\delta \in (0,1)$ 
and $\gep>0$ there are $C,c \in (0,\infty)$ (depending on $\mu,\lambda,\delta,\gep$) for which  
\begin{equation} 
\label{e:ZT}
\mu\prt{\omega \colon \abs{\frac{1}{n} \log E_0^\omega \crt{e^{\lambda H_n}} - \mc{J}^*(\lambda)} > \gep} 
\leq C e^{-c n^{1-\delta}}
\end{equation}
where $\mc{J}^*$ is the Legendre transform of the rate function $\mc{J}$ in Proposition~\ref{htRWRE}.
\end{proposition}

The proof of Proposition~\ref{L:ZT} is given in Section~\ref{proofTn} and is based on an adaptation 
of an argument presented in~\cite{Z04}. We will prove Theorem~\ref{P:MB} by using Proposition~\ref{L:ZT} 
as follows:
\begin{itemize}
\item 
Section~\ref{ss:block}: Partition $[-1,1]$ into blocks and show that Theorem~\ref{P:MB} follows
from a concentration result for each block of the partition.
\item
Section~\ref{ss:halfblock}: Show that concentration on half-lines implies concentration on blocks of 
the partition.  
\item 
Section~\ref{ss:halfhit}: Prove concentration on half-lines by using Proposition~\ref{L:ZT} and the 
relation between hitting times and displacements for RWRE.
\end{itemize}

\subsection{Concentration for hitting times} 
\label{proofTn}

In this section we prove Proposition~\ref{L:ZT}. The proof is based on~\cite[Lemma 3.4.10, p.\ 291]{Z04}.  

Define, for fixed $K \in (0,\infty)$,
\begin{equation}
g^{\delta,n}(\omega) = \log E^\omega_0 \big[e^{\lambda H_n} \Ind{\{H_n < K n\}} 
\Ind{\{N^n < n^{\delta/2}}\big],
\end{equation}
where $N^n = \sup_{x\in\Z} N^n_x$ and $N^n_x$ is the number of visits at $x$ before $H_n$. Note that 
$g^{\delta,n}(\omega)$ is a function of the environment coordinates $(\omega_i\colon\,\abs{i} \leq Kn)$. 
For $i\in\N$, define
\begin{equation}
\begin{aligned}
&\mc{F}_0  = \sigma\chv{\emptyset}, \quad  
\mc{F}_1 = \sigma\chv{\omega_0}, \quad  
\mc{F}_2  = \sigma\chv{\omega_0,\omega_1}, 
\quad\mc{F}_3 = \sigma\chv{\omega_0,\omega_1,\omega_{-1}},\\
&\quad \vdots\\
&\mc{F}_{i} = \sigma\chv{\omega_j \colon\, j \in (- \teto{i/2}, \piso{i/2}] \cap \Z},
\end{aligned}
\end{equation}
and denote by $E^\mu$ expectation with respect to $\mu$. Then
\begin{equation}
E^\mu \crt{g^{\delta,n} \mid \mc{F}_{2Kn}} 
= g^{\delta,n}, \qquad E^\mu \crt{g^{\delta,n} 
\mid \mc{F}_0}  = E^\go_0 \crt{g^{\delta,n}}.
\end{equation}
Rewrite
\begin{equation*}
g^{\delta,n}(\omega) - E^\mu \crt{g^{\delta,n}} = \sum_{i = 1}^{2Kn} d_i(\omega), 
\end{equation*}
with \begin{equation*}
d_i(\omega) = E^\mu \crt{g^{\delta,n} \mid 
\mc{F}_i}(\omega) - E^\mu \crt{g^{\delta,n} \mid \mc{F}_{i-1}}(\omega).
\end{equation*}

Let  $X_0 = E^\mu[g^{\delta,n}]$ and $X_m = \sum_{i = 1}^{m} d_i(\omega)$. Since $E^\mu \crt{d_i 
\mid \mc{F}_m} = 0$ for $i>m$, $\chv{X_m}_{m\in\N_0}$ is a martingale. We obtain a bound for $d_i(\omega) $ 
by writing
\begin{equation}
\label{dif}
\begin{aligned}
d_{i+1}(\omega) = E^\mu \crt{g^{\delta,n}\mid \mc{F}_{i+1}}(\omega) 
- E^\mu \crt{g^{\delta,n}\mid \mc{F}_{i}}(\omega)  
\leq \sup_{\omega^i} \left [g^{\delta,n}(\omega^i) - g^{\delta,n}(\omega)\right]=: |d_i|_{\infty},
\end{aligned}
\end{equation}
where $\omega^i_x = \omega_x$ for all $x\in\Z$, except for
\begin{equation}
x_i =
\begin{cases}
-i/2, & \text{ if $i$ is even},\\
\teto{i/2}, & \text{ if $i$ is odd}.
\end{cases}
\end{equation}
To compute the difference in~\eqref{dif}, we use a bound on the derivative of $g^{\delta,n}(\omega)$. 
From the computations in~\cite[p.\ 291]{Z04} we have that, for any $\delta\in(0,1)$,
\begin{equation}
|d_i|_{\infty}\leq \abs{\frac{\partial g^{\delta,n}(\omega)}{\partial \omega_{x_i}}}  
\leq \sqrt{K}\,\frac{n^{\delta/2}}{\mathfrak{c}}.
\end{equation}
Applying the Azuma-Hoeffding inequality, we obtain
\begin{equation}
\mu\prt{ \go \colon \abs{X_{2Kn} - X_0} > u} \leq 2 \exp \prt{-\frac{u^2}{2 \sum_{i = 1}^{2Kn} |d_i|^2_{\infty}}}.
\end{equation}
Since $X_{2Kn} = g^{\delta,n}(\omega)$ and $X_0 = E^\mu[g^{\delta,n}]$, we obtain
\begin{equation} 
\label{e:AH}
\mu\prt{\go \colon \abs{g^{\delta,n}(\omega) - E^\mu \crt{g^{\delta,n}}} > u n} 
\leq 2 \exp \prt{- \frac{u^2 n^2}{ C n^{1 + \delta}}} \leq 2 \exp\prt{- \frac{u^2}{C} n^{1-\delta}}.
\end{equation}
To conclude the proof of Proposition~\ref{L:ZT}, we write
\begin{equation} \label{H_concent}
\begin{aligned}
&\mu\prt{\omega \colon \abs{\frac{1}{n} 
\log E_0^\omega \crt{e^{\lambda H_n}} - \mc{J}^*(\lambda)} > \gep} 
\leq \mu\prt{\omega \colon \frac{1}{n}\,\abs{\log E_0^\omega\crt{e^{\lambda H_n}} 
- g^{\delta,n}(\omega)} > \tfrac13 \gep}\\
&+ \mu\prt{\omega \colon \frac{1}{n} \abs{g^{\delta,n}(\omega) 
- E^\mu \crt{g^{\delta,n}}} > \tfrac13 \gep} 
+ \mu\prt{\omega \colon \abs{\frac{1}{n}\,E^\mu \crt{g^{\delta,n}} 
- \mc{J}^*(\lambda)} > \tfrac13 \gep}.
\end{aligned}
\end{equation}
We will estimate the second term in the right hand side by~\eqref{e:AH}. Let us first show that the first 
and the third term vanish as $n\to\infty$. For the first term in~\eqref{H_concent}, the ellipticity condition 
implies that for large $n$,
\begin{equation} 
\label{claim_1}
\frac{1}{n}\,\abs{\log E_0^\omega \crt{e^{\lambda H_n}} - g^{\delta,n}(\omega)} < \tfrac13 \gep.
\end{equation}
Indeed, from the argument in the proof of~\cite[Lemma 3.4.10]{Z04}, specifically the computations 
just prior to the statement of~\cite[Lemma 3.4.14]{Z04}, we obtain the following estimate.
For $K = K(\lambda)$ and $n$ large enough, 
\begin{equation}
E^\omega_0 \crt{e^{\lambda H_n} \Ind{\{H_n < K n\}} \Ind{\{N < n^{\delta/2}\}}} 
\geq \tfrac12 E^\omega_0 \crt{e^{\lambda H_n}}.
\end{equation}
Since
\begin{equation}
1\leq \frac{E_0^\omega \crt{e^{\lambda H_n}}}{E^\omega_0 \crt{e^{\lambda H_n} 
\Ind{\{H_n < K n\}} \Ind{\{N < n^{\delta/2}\}}}} \leq 2,
\end{equation}
it follows that
\begin{equation}
\lim_{n\to\infty} \frac{1}{n} \prt{\log E_0^\omega \crt{e^{\lambda H_n}} - g^{\delta,n}(\omega)} = 0,
\end{equation}
which implies~\eqref{claim_1}. Furthermore, since $\lim_{n\to\infty} \frac{1}{n} \log E_0^\omega [e^{\lambda H_n}] 
= \mc{J}^*(\lambda)$,~\eqref{claim_1} also implies that for $n$ large enough the third term in~\eqref{H_concent} 
is zero. We conclude that for $n$ large enough,
\begin{equation}
\begin{aligned}
\mu\prt{\omega \colon \abs{\frac{1}{n} \log E_0^\omega\crt{e^{\lambda H_n}} 
- \mc{J}^*(\lambda)} > \gep} \leq \mu \prt{\omega \colon \frac{1}{n} \abs{g^{\delta,n}(\omega) 
- E^\go_0 \crt{g^{\delta,n}}} > \tfrac13 \gep}.
\end{aligned}
\end{equation}
Picking $u =\tfrac13 \gep$ in~\eqref{e:AH}, we obtain~\eqref{e:ZT}.

\subsection{Block decomposition} 
\label{ss:block}

Note that $\frac{Z_n}{n} \in [-1,1]$. Consider the following block decomposition (see Fig.~\ref{fig:Blockdepomp}):
\begin{equation} \label{blocks}
\gD^N_i =
\begin{cases}
[- 1, - 1 + \frac{1}{N}], &\text{ if } i = - N, \\ 
(\frac{i}{N}, \frac{i+1}{N}], &\text{ if } i \in \chv{- N + 1, \ldots, N - 1}.
\end{cases}
\end{equation}
To deal with the flat piece of the rate function $\mc{I}$ in the positive-speed case, we define the following 
interval $\Delta^{*,N}_0$ containing $(0,v_\mu]$ (see Fig.~\ref{fig:Blockdepomp}):
\begin{equation}
\Delta^{*,N}_0 = \Big(0, \frac{\piso{v_\mu N + 1}}{N}\Big] 
= \bigcup_i\, \left\{\Delta^{N}_i\colon\, \mc{I}\left(\frac{i-1}{N}\right)=0\right\}, 
\qquad \Delta^{*,N}_i = \gD^{N}_{i} \setminus \Delta^{*,N}_0.
\end{equation}

\begin{figure}[htbp]
\begin{center}
\includegraphics[clip, trim=3cm 10cm 3cm 14cm, width=0.8\textwidth]{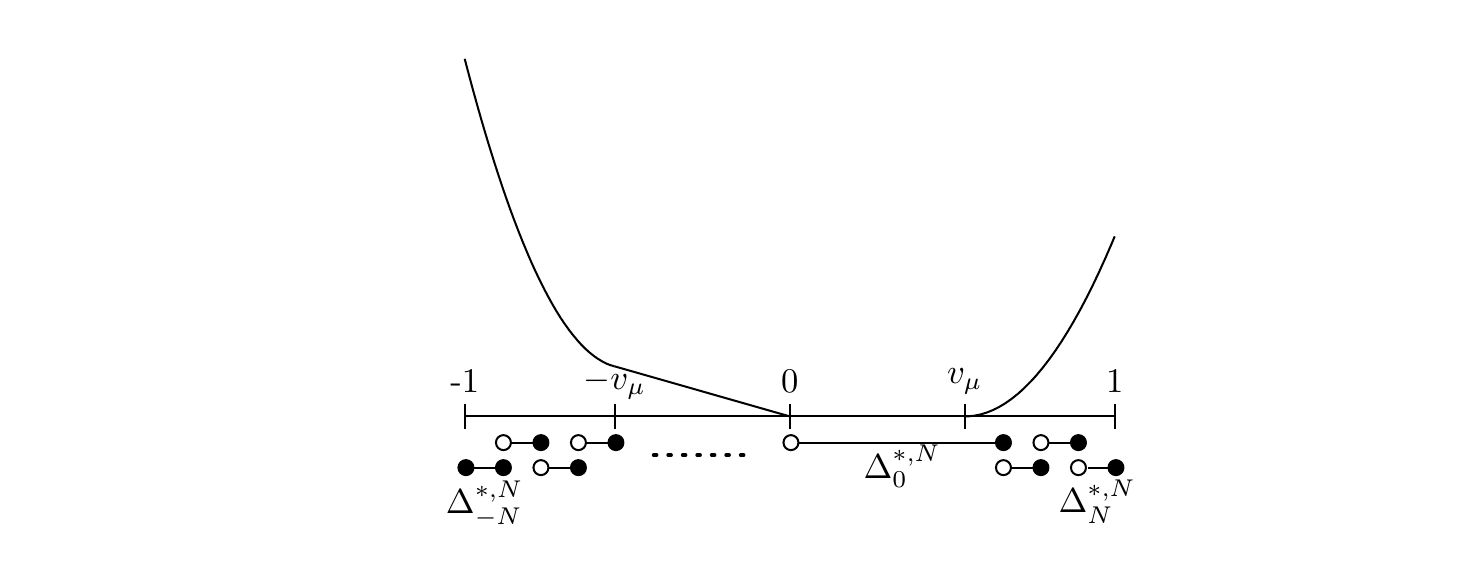}
\caption{Block decomposition of $[-1,1]$. Black and white circles indicate closed and open 
boundaries of the intervals, respectively. All the intervals are of length $1/N$, except possibly 
$\Delta^{*,N}_0$, which contains the flat piece $(0,v_\mu]$.}
\label{fig:Blockdepomp}
\end{center}
\end{figure}

\noindent
By the intermediate value theorem, we have
\begin{equation}
E^\go_0 \crt{e^{n \gl \frac{Z_n}{n}} \Ind{\crl{\frac{Z_n}{n} \in \gD_i^{*,N}}}} 
= e^{n \gl u_i^*} P^\omega_0 \prt{\frac{Z_n}{n} \in \gD_i^{*,N}}
\end{equation}
for some $u_i^* \in \gD_i^{*,N}$. Now,
\begin{equation}
\begin{aligned}
\frac{1}{n} \log E^\omega_0 \crt{e^{\lambda Z_n}} - \mc{I}^*(\gl) 
&= \frac{1}{n} \log e^{- n \mc{I}^*(\gl)} \sum_i E^\go_0 \crt{e^{n \gl \frac{Z_n}{n}} 
\Ind{\crl{\frac{Z_n}{n} \in \gD_i^{*,N}}}}\\
&= \frac{1}{n} \log \sum_i e^{n \gl u_i^* - n \mc{I}^\go_n(\gD_i^{*,N}) - n \mc{I}^*(\gl)},
\end{aligned}
\end{equation}
where $\mc{I}^\go_n(\gD) = - \frac{1}{n} \log P^\omega_0 \left(\frac{Z_n}{n}  \in \gD\right)$. 
Since $\mc{I}^\go_n(\gD)$ converges to $\mc{I}(\gD)$ as $n\to\infty$, we define the 
block error
\begin{equation} 
\label{error1}
o(n,\gD,\go) = \mc{I}(\gD) - \mc{I}^\go_n(\gD)
\end{equation}
and obtain
\begin{equation} 
\label{d_CGF}
\frac{1}{n} \log E^\omega_0 \crt{e^{\lambda Z_n}} - \mc{I}^*(\gl)
= \frac{1}{n} \log \sum_i e^{n \prt{\gl u_i^* - \mc{I}(\gD_i^{*,N}) 
- \mc{I}^*(\gl) + o(n,\gD_i^{*,N},\go)}}.
\end{equation}

To  estimate~\eqref{d_CGF}, we will need the following lemma.

\begin{lemma}[{\bf Reduction to the worst block}] \label{l:blocks}
Given $\gep > 0$, there is an $N_0$ such that, for $N > N_0$ and $n > n_0(N)$,
\begin{equation} 
\label{eq:bound_1}
\abs{\frac{1}{n} \log E^\omega_0 \crt{e^{\lambda Z_n}} - \mc{I}^*(\gl)} 
\leq \tfrac12 \gep + \max_{i} \abs{o(n,\gD_i^{*,N},\go)}
\end{equation}
with $o(n,\gD_i^{*,N},\go)$ as in~\eqref{error1}.
\end{lemma}

\begin{proof} 
Since $\mc{I}$ is uniformly continuous in $[-1,1]$, we have
\begin{equation}
\gd_N = \sup_{\abs{s-t}\leq 1/N} \abs{\mc{I}(s) - \mc{I}(t)} \to 0, \qquad N \to \infty.
\end{equation}
Let $ \gd_i = \mc{I}(u_i^*) -\mc{I}(\gD_i^{*,N})$, and note that $0 < \delta_i \leq \delta_N$.

\paragraph{Upper bound.}

Since
\begin{equation}
\mc{I}^*(\gl) = \sup_{u\in\R} \big[\gl u - \mc{I}(u)\big] \geq \gl u_i^* - \mc{I}(u_i^*),
\end{equation}
we get the bound
\begin{equation}
\gl u_i^* - \mc{I}(\gD_i) - \mc{I}^*(\gl) \leq \gl u_i^* - \big[\mc{I}(u_i^*) - \gd_i\big] 
- \big[\gl u_i^* - \mc{I}(u_i^*)\big] = \gd_i < \gd_N.
\end{equation}
Let $N_0$ be such that $\delta_{N_0} < \tfrac12\gep$. For $N > N_0$, let $n_0(N)$ be such that 
$\gd_{N_0} + \frac{\log 2N}{n_0} < \tfrac12\gep$.  For $n > n_0(N)$, we have
\begin{equation}
\label{e:upper}
\begin{aligned}
\frac{1}{n} \log E^\omega_0 \crt{e^{\lambda Z_n}} - \mc{I}^*(\gl) 
&\leq \frac{1}{n} \log \sum_i e^{n \prt{\gd_N + o(n,\gD_i^{*,N},\go)}}\\
&\leq \gd_N + \frac{\log 2N}{n} + \max_i o(n,\gD_i^{*,N},\go)\\
&\leq \tfrac12 \gep + \max_i o(n,\gD_i^{*,N},\go).
\end{aligned}
\end{equation}

\paragraph{Lower bound.}

Let $\hat{x}$ be such that $\mc{I}^*(\gl) = \gl \hat{x} - \mc{I}(\hat{x})$. Then there exists a 
$\hat{\gi}$ such that $\hat{x} \in \gD_{\hat{\gi}}^{\ast,N}$. For large $N$, we see that 
$\hat{x} \notin \Delta^{*,N}_0$. Indeed, if $\lambda < 0$, then $\hat{x} \leq 0$ and therefore 
$\hat{x} \notin \Delta^{*,N}_0$. On the other hand, if $\lambda > 0$, then $\hat{x} > v_\mu$. 
Pick $N$ large enough so that $v_\mu + 1/N < \hat{x}$. Since $\Delta^{*,N}_0  \subset 
(0, v_\mu + 1/N]$, we conclude that $\hat{x} \notin \Delta^{*,N}_0$.

Since $\hat{x} \in \Delta^{*,N}_{\hat{\iota}}$, it follows that $\abs{\hat{x} - u^*_{\hat{\gi}}} < 1/N$.  
Choosing $N_0$ so that $\big| \frac{\gl}{N_0} + \gd_{N_0} \big| < \tfrac12 \gep$, we obtain
\begin{equation}
\label{e:lower}
\begin{aligned}
\frac{1}{n} \log E^\omega_0 \crt{e^{\lambda Z_n}} - \mc{I}^*(\gl) 
&\geq \frac{1}{n} \log e^{n\left([\gl u^*_{\hat{\gi}} - \mc{I}(\Delta^{*,N}_{\hat{\gi}})] 
- [\gl \hat{x} - \mc{I}(\hat{x})] + o(n,\gD_{\hat{\gi}}^{*,N},\go)\right)}\\
&\geq - \prt{\frac{\gl}{N} + \gd_N} + o(n,\gD_{\hat{\gi}}^{*,N},\go)\\
&\geq - \tfrac12 \gep + o(n,\gD_{\hat{\gi}}^{*,N},\go).
\end{aligned}
\end{equation}
The claim in~\eqref{eq:bound_1} follows from~\eqref{e:upper} and~\eqref{e:lower}.
\end{proof}

In view of Lemma~\ref{l:blocks}, we can bound
\begin{equation}
\begin{aligned}
\mu\prt{\go \colon \abs{\frac{1}{n} \log E^\omega_0 \crt{e^{\lambda Z_n}} - \mc{I}^*(\gl)} > \gep} 
&\leq \mu\prt{\go \colon \max_i \abs{o(n,\gD_i^{*,N},\go)} > \tfrac12 \gep}\\
&\leq \sum_i \mu\prt{\go \colon \abs{o(n,\gD_i^{*,N},\go)} > \tfrac12 \gep}.
\end{aligned}
\end{equation}
Since there are only finitely many terms in the sum,~\eqref{e:MB} follows once we prove the following:

\begin{lemma}[{\bf Concentration of empirical speed on an interval}] \label{L:Dt}
Let $\mu$ be basic, $\delta > 0$, and let $\gD = (a,b]$ or $\gD = [a,b]$ with $a \neq b$.
\begin{itemize}
\item 
If $0 \leq a$, then for every $\gep>0$ there are positive constants $C,c$ such that
\begin{equation} 
\label{e:Dt+}
\mu\prt{\omega \colon \abs{\frac{1}{n} \log P_0^\omega \prt{\frac{Z_n}{n} \in \gD} 
 +\mc{I}(a) } > \gep} \leq C \, e^{- c n^{1-\delta}}.
\end{equation}
\item 
If $b \leq 0$, then for every $\gep>0$ there are positive constants $C,c$ such that
\begin{equation} 
\label{e:Dt-}
\mu\prt{\omega \colon \abs{\frac{1}{n} \log P_0^\omega \prt{\frac{Z_n}{n} \in \gD} 
+\mc{I}(b) } > \gep}  \leq C \, e^{- c n^{1-\delta}}.
\end{equation}
\end{itemize}
\end{lemma}

The proof of Lemma~\ref{L:Dt} will be given in the next section as a corollary to the following 
concentration result for the half-line:

\begin{lemma}[{\bf Concentration of empirical speed on a half-line}] \label{L:*}
Let $\mu$ be basic and $\delta>0$.
\begin{itemize}
\item 
If $0 \leq a$, then for every $\gep>0$ there are positive constants $C,c$ such that
\begin{equation} 
\label{e:BSI+}
\mu\prt{\omega \colon \abs{\frac{1}{n} \log P_0^\omega \prt{\frac{Z_n}{n} > a} + \mc{I}(a)} > \gep} 
\leq C \, e^{- c n^{1-\delta}}. 
\end{equation}
\item 
If $b \leq 0$, then for every $\gep>0$ there are positive constants $C,c$ such that
\begin{equation} 
\label{e:BSI-}
\mu\prt{\omega \colon \abs{\frac{1}{n} \log P_0^\omega \prt{\frac{Z_n}{n} \leq b} + \mc{I}(b)} > \gep} 
\leq C \, e^{- c n^{1-\delta}}.
\end{equation}
\end{itemize}
\end{lemma}

\subsection{Concentration: from half-lines to intervals}
\label{ss:halfblock}

In this section we prove that Lemma~\ref{L:Dt} follows from Lemma~\ref{L:*}.

\begin{proof}
We will split the proof for $\Delta = (a,b]$ in two cases. The proof for $\gD = [a,b]$ is similar.

\paragraph{Case $0 \leq a$.}

We start from the equation
\begin{equation}
P^\omega_0 \prt{\frac{Z_n}{n} \in \gD} = P^\omega_0 \prt{\frac{Z_n}{n} > a} - P^\omega_0 \prt{\frac{Z_n}{n} > b}.
\end{equation}
Define $e(\omega,u,n) = \frac{1}{n} \log P_0^\omega(\frac{Z_n}{n} > u) + \mc{I}(u)$. Since $\mc{I}(b) - \mc{I}(a) 
= \eta > 0$, we obtain
\begin{equation}
\begin{aligned}
\frac{P^\omega_0 \prt{\frac{Z_n}{n} > a}}{P^\omega_0 \prt{\frac{Z_n}{n} > b}} 
&= e^{n (\mc{I}(b) - \mc{I}(a)  + e(\omega,a,n) - e(\omega,b,n))} = e^{n (\eta + e(\omega,a,n) - e(\omega,b,n))}.
\end{aligned}
\end{equation}
When $\abs{e(\omega,a,n) - e(\omega,b,n)} < \tfrac12 \eta$, we have
\begin{equation} 
\label{e:eta2}
P^\omega_0 \prt{\frac{Z_n}{n} > b} \leq e^{- n \frac{\eta}{2}} \, P^\omega_0 \prt{\frac{Z_n}{n} > a}.
\end{equation}
For large enough $n$, as soon as $e^{- n \frac{\eta}{2}} < \tfrac{1}{2}$ we get
\begin{equation}
\tfrac{1}{2} P^\omega_0 \prt{\frac{Z_n}{n} > a} \leq P^\omega_0 \prt{\frac{Z_n}{n} \in \gD} 
\leq P^\omega_0 \prt{\frac{Z_n}{n} > a},
\end{equation}
which implies that
\begin{equation} 
\label{e:log2}
\abs{\frac{1}{n} \log P^\omega_0 \prt{\frac{Z_n}{n} \in \gD} 
- \frac{1}{n} \log P^\omega_0 \prt{\frac{Z_n}{n} > a}} < \frac{1}{n}\log 2.
\end{equation} 
Therefore
\begin{equation} \label{inter_a}
\begin{aligned}
&\mu \prt{\omega \colon \abs{\frac{1}{n} \log P_0^\omega \prt{\frac{Z_n}{n} \in \gD} 
- \frac{1}{n} \log P^\omega_0 \prt{\frac{Z_n}{n} > a}} > \frac{1}{n}\log 2}\\
&\leq \mu \prt{\omega \colon \abs{e(\omega,a,n) - e(\omega,b,n)} > \tfrac12\eta}\\[0.2cm]
&\leq \mu \prt{\omega \colon \abs{e(\omega,a,n)} > \tfrac14\eta} 
+ \mu \prt{\omega \colon \abs{e(\omega,b,n)} > \tfrac14\eta}.
\end{aligned}
\end{equation}
Since the concentration~\eqref{e:BSI+} in Lemma~\ref{L:*} bounds both terms in~\eqref{inter_a}, 
after we pick $n$ large enough so that $\frac{\log 2}{n}<\gep$, we obtain~\eqref{e:Dt+}.

\paragraph{Case $b\leq 0$.}

In this case we have the following equation:
\begin{equation}
P^\omega_0 \prt{\frac{Z_n}{n} \in \gD} 
= P^\omega_0 \prt{\frac{Z_n}{n} \leq b} - P^\omega_0 \prt{\frac{Z_n}{n} \leq a}.
\end{equation}
Similarly, we define $\tilde{e}(\omega,u,n) = \frac{1}{n} \log P^\omega_0(\frac{Z_n}{n} \leq u) 
+ \mc{I}(u)$.  Since $\mc{I}(a) - \mc{I}(b) = \eta > 0$, we obtain
\begin{equation}
\begin{aligned}
\frac{P^\omega_0 \prt{\frac{Z_n}{n} \leq b}}{P^\omega_0 \prt{\frac{Z_n}{n} \leq a}} 
= e^{n (\mc{I}(a) - \mc{I}(b) - \tilde{e}(\omega,a,n) + \tilde{e}(\omega,b,n))} 
= e^{n(\eta + \tilde{e}(\omega,a,n) - \tilde{e}(\omega,b,n))}.
\end{aligned}
\end{equation}
When $\abs{\tilde{e}(\omega,a,n) - \tilde{e}(\omega,b,n)} < \tfrac12\eta$,
\begin{equation}
P^\omega_0 \prt{\frac{Z_n}{n} \leq a} \leq e^{-n \frac{\eta}{2}} \, P^\omega_0 \prt{\frac{Z_n}{n} \leq b}.
\end{equation}
As we did in~\eqref{e:eta2}--\eqref{e:log2}, for large $n$ as soon as $e^{- n \frac{\eta}{2}} < \frac{1}{2}$
we conclude that
\begin{equation} \label{inter_b}
\begin{aligned}
&\mu \prt{\omega \colon \abs{\frac{1}{n} \log P_0^\omega \prt{\frac{Z_n}{n} \in \gD} 
- \frac{1}{n} \log P^\omega_0 \prt{\frac{Z_n}{n} \leq b}} > \tfrac{1}{n}\log 2}\\
&\leq \mu \prt{\omega \colon \abs{\tilde{e}(\omega,a,n) - \tilde{e}(\omega,b,n)} > \tfrac12\eta}\\[0.2cm]
&\leq \mu \prt{\omega \colon \abs{\tilde{e}(\omega,a,n)} > \tfrac14\eta} 
+ \mu \prt{\omega \colon \abs{\tilde{e}(\omega,b,n)} > \tfrac14\eta}.
\end{aligned}
\end{equation}
Since the concentration~\eqref{e:BSI-} in Lemma~\ref{L:*} bounds both terms in~\eqref{inter_b}, 
after we pick $n$ large enough so that $\frac{\log 2}{n}<\gep$, we obtain~\eqref{e:Dt-} and 
Lemma~\ref{L:Dt} follows.
\end{proof}

\subsection{Concentration: from hitting times to half-lines}
\label{ss:halfhit}

In this section we prove Lemma~\ref{L:*}.

\begin{proof}
Once we prove~\eqref{e:BSI+}, the proof of~\eqref{e:BSI-} follows from a reflection argument.
Indeed, let $\widetilde{\omega} = \prt{\widetilde{\omega}(x)}_{x \in \Z} = \prt{1 - \omega(x)}_{x \in \Z}$. 
For $u > 0$, $P^\omega_0(\frac{Z_n}{n} < - u) = P^{\widetilde{\omega}}_0(\frac{Z_n}{n} > u)$ and, 
denoting by $\mc{I}^\go$ the quenched rate function on $\go$, we get
\begin{equation}
\mc{I}^{\widetilde{\omega}}(u) = \mc{I}^\omega(- u).
\end{equation}
Therefore
\begin{equation}
\frac{1}{n} \log P^\omega_0 \prt{\frac{Z_n}{n} < - u} + \mc{I}^\omega(- u) 
= \frac{1}{n} \log P^{\widetilde{\omega}}_0 \prt{\frac{Z_n}{n} > u} + \mc{I}^{\widetilde{\omega}}(u)
\end{equation}
and
\begin{equation}
\begin{aligned}
&\mu \prt{\omega \colon \abs{\frac{1}{n} \log P^\omega_0 \prt{\frac{Z_n}{n} < - u} 
+ \mc{I}^\omega(- u)} > \gep}\\
&= \mu \prt{\omega \colon \abs{\frac{1}{n} \log P^{\widetilde{\omega}}_0 \prt{\frac{Z_n}{n} > u} 
+ \mc{I}^{\widetilde{\omega}}(u)} > \gep}\\
&= \tilde{\mu} \prt{\omega \colon \abs{\frac{1}{n} \log P^\omega_0 \prt{\frac{Z_n}{n} > u} 
+ \mc{I}^\omega(u)} > \gep},\\
\end{aligned}
\end{equation}
where $\tilde{\mu}\crt{\omega \in A} = \mu\crt{\widetilde{\omega}\in A}$ satisfies the conditions of 
Lemma~\ref{L:*}. After proving~\eqref{e:BSI+} for $\mu$, we obtain~\eqref{e:BSI+} for $\tilde{\mu}$, 
which is equivalent to the proof of~\eqref{e:BSI-} for $\mu$. 
 
To prove~\eqref{e:BSI+}, we derive upper and lower bounds for $\frac{1}{n} \log P^{\omega}_0
(\frac{Z_n}{n} \geq u) + \mc{I}(u)$.

\paragraph{Upper bound.}

To bound the probabilities on the displacements we can use the hitting times. For $u > 0$,
\begin{equation}
P^\omega_0 \prt{\frac{Z_n}{n} \geq u} \leq P^\omega_0 \prt{Z_n \geq \piso{u n}} 
\leq P^{\omega}_0 \prt{H_{\piso{u n}} \leq n}.
\end{equation}
By the Markov inequality, for $\gt < 0$,
\begin{equation}
P^\omega_0 \prt{H_{\piso{u n}}\leq n} \leq e^{- \gt n} E^\go_0 \crt{e^{\gt H_{\piso{u n}}}},
\end{equation}
and so
\begin{equation}
\begin{aligned}
\frac{1}{n} \log P_0 \prt{\frac{Z_n}{n} \geq u} 
&\leq - \gt + \frac{1}{n} \log E^\go_0 \crt{e^{\gt H_{\piso{u n}}}}\\
&= - \gt + \frac{1}{n} \log E^\go_0 \crt{e^{\gt H_{\piso{u n}}}} + u \mc{J}^*(\gt) - u \mc{J}^*(\gt)\\
&= - u \prt{\gt \frac{1}{u} + O(\go,u n,\gt) - \mc{J}^*(\gt)},
\end{aligned}
\end{equation}
where
\begin{equation}
O(\go,u n, \gt) = \mc{J}^*(\gt) - \frac{1}{u n} \log E^\go_0 \crt{e^{\gt H_{\piso{u n}}}}.
\end{equation}
Taking $\gt < 0$ such that $\gt \frac{1}{u} - \mc{J}^*(\gt) = \mc{J}(\frac{1}{u})$ (see Fig.~\ref{fig:theta}) 
and using that $u \mc{J}(\frac{1}{u}) = \mc{I}(u)$, we get
\begin{equation}
\frac{1}{n} \log P^\go_0 \prt{\frac{Z_n}{n} \geq u} + \mc{I}(u) \leq - u \, O(\go,u n,\gt).
\end{equation}

\begin{figure}[htbp]
\vspace{-.5cm}
\begin{center}
\includegraphics[clip, trim=3cm 10cm 3cm 12cm, width=0.6\textwidth]{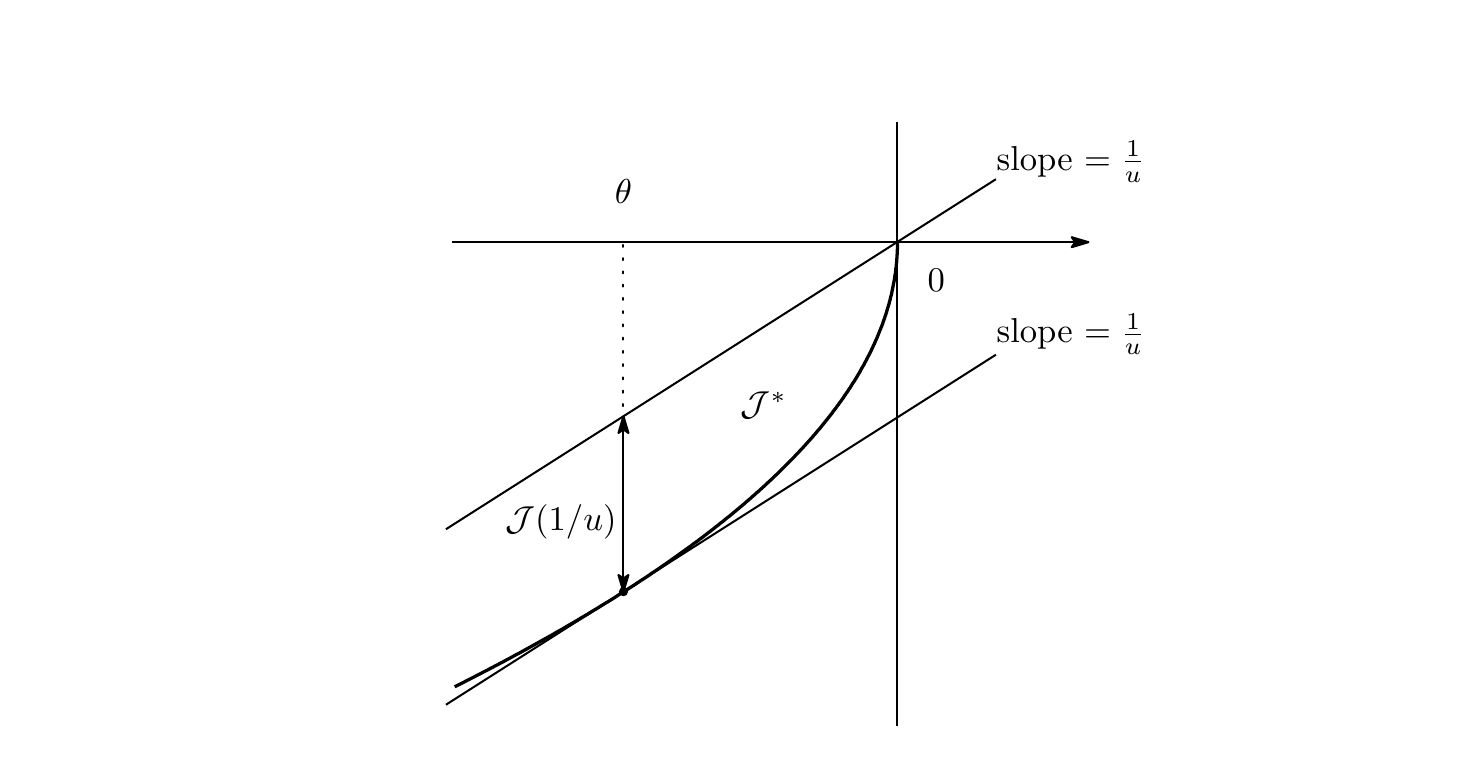}
\caption{The function $x\mapsto \frac{1}{u} x - \mc{J}^*(x)$ attains its maximum at $\theta$.}
\label{fig:theta}
\end{center}
\end{figure}
Therefore, using Proposition~\ref{L:ZT}, we arrive at
\begin{equation} 
\label{e:UB*}
\mu \prt{\go \colon \frac{1}{n} \log P^\go_0 \prt{\frac{Z_n}{n} \geq u} + \mc{I}(u) >\gep} 
\leq \mu \prt{\go \colon \abs{O(\go,u n,\gt)} > \gep} < C e^{- c n^{1-\delta}}.
\end{equation}

\paragraph{Lower bound.} 

The lower bound for $P^\omega_0\left(\frac{Z_n}{n} \geq u\right)$ is more subtle. Note that, since the steps 
of the random walk are either $+ 1$ or $- 1$, for $d > 0$, we have 
\begin{equation}
n < H_x < n + dn \quad \Longrightarrow \quad Z_n > x - dn.
\end{equation}
Therefore,
\begin{equation}
P^{\omega}_0 \prt{\frac{Z_n}{n} \geq u} \geq P^\omega_0 \prt{Z_n \geq \teto{u n}} 
\geq P^\omega_0 \prt{n \leq H_{\teto{u n} + \piso{d n}} \leq n + \piso{d n}}.
\end{equation}
Now, let $m = \teto{u n} + \piso{d n}$. Note that
\begin{align*}
\frac{1}{u + d + r_n} \leq \frac{H_m}{m} \leq \frac{1 + d + \tilde{r}_n}{u + d + r_n}
\end{align*}
with $r_n, \tilde{r}_n \to 0$ as $n \to 0$. Let $\tilde{d}$ and $\tilde{u}$ be such that 
\begin{equation}
\frac{1}{u + d} < \frac{1}{\tilde{u}} - \tilde{d} < \frac{1}{\tilde{u}} - \tilde{d} < \frac{1 + d}{u + d}.
\end{equation}
Letting $B_{\tilde{d}}(1/\tilde{u})$ denote the ball with center $\frac{1}{\tilde{u}}$ and radius 
$\tilde{d}$, we have, for $n$ large enough,
\begin{equation} \label{lhalfball}
\frac{1}{n} \log P^\omega_0 \prt{\frac{Z_n}{n} \geq u} 
\geq \frac{1}{n} \log P^{\omega}_0 \prt{\frac{H_m}{m} \in B_{\tilde{d}}(1/\tilde{u})}.
\end{equation}
If $d \to 0$, then $\abs{\tilde{u} - u} \to 0$ and $\tilde{d} \to 0$. Note that $E^\go_0[e^{\zeta H_m}] 
< \infty$ for $\zeta < 0$. We define the $\zeta$-tilted probability measure
\begin{equation}
\frac{\dd P^{\omega,\zeta,m}_{0}}{\dd P^{\omega,m}_0}(y) 
= \frac{e^{m \zeta y}}{E^\go_0 \crt{e^{\zeta H_m}}}, \qquad 
P^{\omega,m}_{0}(\cdot) = P^\omega_0 \prt{\frac{H_m}{m} \in \cdot}.
\end{equation}
Recalling that $E^\go_0 \crt{e^{\zeta H_m}} = \int e^{m \zeta y} \dd P^{\omega,m}_{0}(y)$, we compute
\begin{equation} 
\label{e:tilt}
\begin{aligned}
\frac{1}{n} \log P^{\omega}_0 \prt{\frac{H_m}{m} \in B_{\tilde{d}}(1/\tilde{u})} 
&= \frac{1}{n} \log \int_{B_{\tilde{d}}(1/\tilde{u})} \dd P^{\omega,m}_{0}(y)\\
&= \frac{1}{n} \log \int_{B_{\tilde{d}}(1/\tilde{u})} 
\frac{E^\go_0 \crt{e^{\zeta H_m}}}{e^{m \zeta y}} \dd P^{\omega,\zeta,m}_{0}(y).
\end{aligned}
\end{equation}
Now, since $\zeta < 0$,
\begin{equation}
y \in B_{\tilde{d}}(1/\tilde{u}) \quad \Longrightarrow \quad 
e^{- m \zeta y} \geq e^{-m \zeta \prt{\frac{1}{\tilde{u}} - \tilde{d}}}.
\end{equation}
Inserting this into~\eqref{e:tilt} and replacing $\frac{1}{n}\log E^\go_0 \crt{e^{\zeta H_m}}$ by 
$- \frac{m}{n} [O(\go,m,\zeta) -  \mc{J}^*(\zeta)]$, yields:
\begin{equation}
\begin{aligned} \label{balltilt}
&\frac{1}{n} \log P^\omega_0 \prt{\frac{H_m}{m} \in B_{\tilde{d}}(1/\tilde{u})}\\
&\geq \frac{1}{n} \log E^\go_0 \crt{e^{\zeta H_m}} - \frac{m}{n} \zeta \prt{\frac{1}{\tilde{u}} 
- \tilde{d}} + \frac{1}{n} \log P^{\omega,\zeta,m}_{0}(B_{\tilde{d}}(1/\tilde{u}))\\
&= - \hat{u}_n \crt{ \prt{\zeta \frac{1}{\hat{u}_n} - \mc{J}^*(\zeta)} + O(\go,m,\zeta)} 
- \hat{d} \zeta + \frac{1}{n} \log P^{\omega,\zeta,m}_{0} (B_{\tilde{d}}(1/\tilde{u})),
\end{aligned}
\end{equation}
where $\hat{u}_n =  \frac{m}{n}$ and $\hat{d}$ is defined by the relation $\frac{m}{n} 
\prt{\frac{1}{\tilde{u}} - \tilde{d}}  = 1 + \hat{d}$. Note that
\begin{equation} \label{e:convd}
 \hat{u}_n \to u, \hat{d} \to 0 \quad \text { as }\quad d \to 0, n \to \infty
\end{equation}
Since $\zeta \frac{1}{\hat{u}_n} - \mc{J}^*(\zeta) \leq \mc{J}(\frac{1}{\hat{u}_n})$, 
combining~\eqref{balltilt} with~\eqref{lhalfball} we obtain
\begin{equation}
\begin{aligned}
\frac{1}{n} \log P^\go_0 \prt{\frac{Z_n}{n} \geq u} + \mc{I}(u) 
&\geq \mc{I}(u) - \mc{I}(\hat{u}_n) - \hat{u}_n \, O(\go,m,\zeta) -\hat{d} \zeta\\
&\qquad + \frac{1}{n} \log P^{\omega,\zeta,m}_{0} (B_{\tilde{d}}(1/\tilde{u})).
\end{aligned}
\end{equation}
Therefore, taking $d$ small enough and $n$ large enough so that
\begin{equation}\label{e:halfI}
  \abs{\mc{I}(u) - \mc{I}(\hat{u}_n)}   + |\hat{d} \zeta| < \tfrac12\gep,
\end{equation}
we get
\begin{equation} 
\label{last1}
\begin{aligned}
&\mu \prt{\go \colon \frac{1}{n} \log P^\go_0 \prt{\frac{Z_n}{n} \geq u} + \mc{I}(u) < - \gep}\\
&\leq \mu \prt{\go \colon -\hat{u}_n \, O(\go,m,\zeta) 
+ \frac{1}{n} \log P^{\omega,\zeta,m}_0 (B_{\tilde{d}}(1/\tilde{u})) < - \tfrac12\gep}\\
&\leq \mu \prt{\go \colon -\hat{u}_n \, O(\go,m,\zeta)  <- \tfrac14\gep} 
+ \mu \prt{\go \colon \frac{1}{n} \log P^{\omega,\zeta,m}_0 (B_{\tilde{d}}(1/\tilde{u})) 
< - \tfrac14\gep}.
\end{aligned}
\end{equation}
From Proposition~\ref{L:ZT} and the fact that $0<\hat{u}_n \leq 1$, it follows that
\begin{equation} 
\label{last2}
\mu \prt{\go \colon \hat{u}_n \, \abs{O(\go,m,\zeta)} > \tfrac14 \gep} 
\leq C \, e^{- c n^{1-\delta}}.
\end{equation}
It therefore remains to prove that
\begin{equation}
\mu \prt{\go \colon \frac{1}{n} \log P^{\omega,\zeta,m}_0 (B_{\tilde{c}}(1/\tilde{u})) < - \tfrac14\gep} 
\leq C \, e^{- c n^{1-\delta}}.
\end{equation}

Let $E^{\go,\zeta,m}_0[f(Y)]$ be
expectation of $f$ with respect to $P^{\omega,\zeta,m}_0(dY)$  and $E^{\omega,m}_0 [f(Y)]$ be
expectation of $f$ with respect to $P^{\omega,m}_0(dY)$. Then
\begin{equation}
E^{\go,\zeta,m}_0 \crt{e^{m \gt Y}} 
= \frac{E^{\go,m}_0 \crt{e^{m (\gt + \zeta) Y}}}{E^{\go,m}_0 \crt{e^{m \zeta Y}}} 
= \frac{E^{\go}_0 \crt{e^{(\gt + \zeta) H_m}}}{E^\go_0 \crt{e^{\zeta H_m}}}.
\end{equation}
We have
\begin{equation}
\begin{aligned}
P^{\omega,\zeta,m}_0 (B_{\tilde{d}}(1/\tilde{u})^c) 
&= \underbrace{P^{\omega,\zeta,m}_0 \prt{Y > \frac{1}{\tilde{u}} + \tilde{d}}} 
+ \underbrace{P^{\omega,\zeta,m}_0 \prt{Y < \frac{1}{\tilde{u}} - \tilde{d}}}\\
&= \qquad\qquad I \qquad \qquad \;\;\,+ \qquad\quad II.
\end{aligned}
\end{equation}

\paragraph{Bound for $I$.}

For $\gt > 0$,
\begin{equation}
\begin{aligned}
&P^{\omega,\zeta,m}_0 \prt{Y > \frac{1}{\tilde{u}} + \tilde{d}} 
\leq e^{- m \gt (\frac{1}{\tilde{u}} + \tilde{d})} E^{\go,\zeta,m}_0 \crt{e^{\gt Y}}\\
&= \exp \crl{m \crt{- \gt \left(\frac{1}{\tilde{u}} + \tilde{d}\right) 
+ \frac{1}{m} \prt{\log E^\go_0 \crt{e^{ (\gt + \zeta) H_m}} 
- \log E^\go_0 \crt{e^{\zeta H_m}}}}}\\
&\leq \exp \crl{m \crt{- \gt \left(\frac{1}{\tilde{u}} + \tilde{d}\right) + \mc{J}^*(\gt + \zeta) 
- \mc{J}^*(\zeta) - O(\go,m,\gt+\zeta) + O(\go,m,\zeta)}}.
\end{aligned}
\end{equation}
By the strict convexity of $\mc{J}$ at $\frac{1}{\tilde{u}} > 0$, we can pick $\zeta = \mc{J}'(\frac{1}{\tilde{u}}) 
< 0$ an exposing plane, i.e., $\zeta$ is such that for any $y \neq \frac{1}{\tilde{u}}$,
\begin{equation} 
\label{eq:exposing}
\mc{J}(y)- \mc{J}\left(\frac{1}{\tilde{u}}\right) > \prt{y - \frac{1}{\tilde{u}}} \zeta.
\end{equation}
By the strict convexity of $\mc{J}$ and the fact that $\frac{1}{\tilde{u}} + \tilde{d} > \frac{1}{\tilde{u}}$, 
we can pick $\gt > 0$ (see Fig.~\ref{fig:exposing_J_J*}) such that
\begin{equation} 
\label{e:theta}
(\gt + \zeta) \left(\frac{1}{\tilde{u}}+ \tilde{d}\right) - \mc{J}\left(\frac{1}{\tilde{u}} + \tilde{d}\right) 
= \mc{J}^*(\gt + \zeta).
\end{equation}

\begin{figure}[htbp]
\begin{center}
\includegraphics[clip, trim=2.8cm 10.5cm 2cm 12.5cm, width=0.8\textwidth]{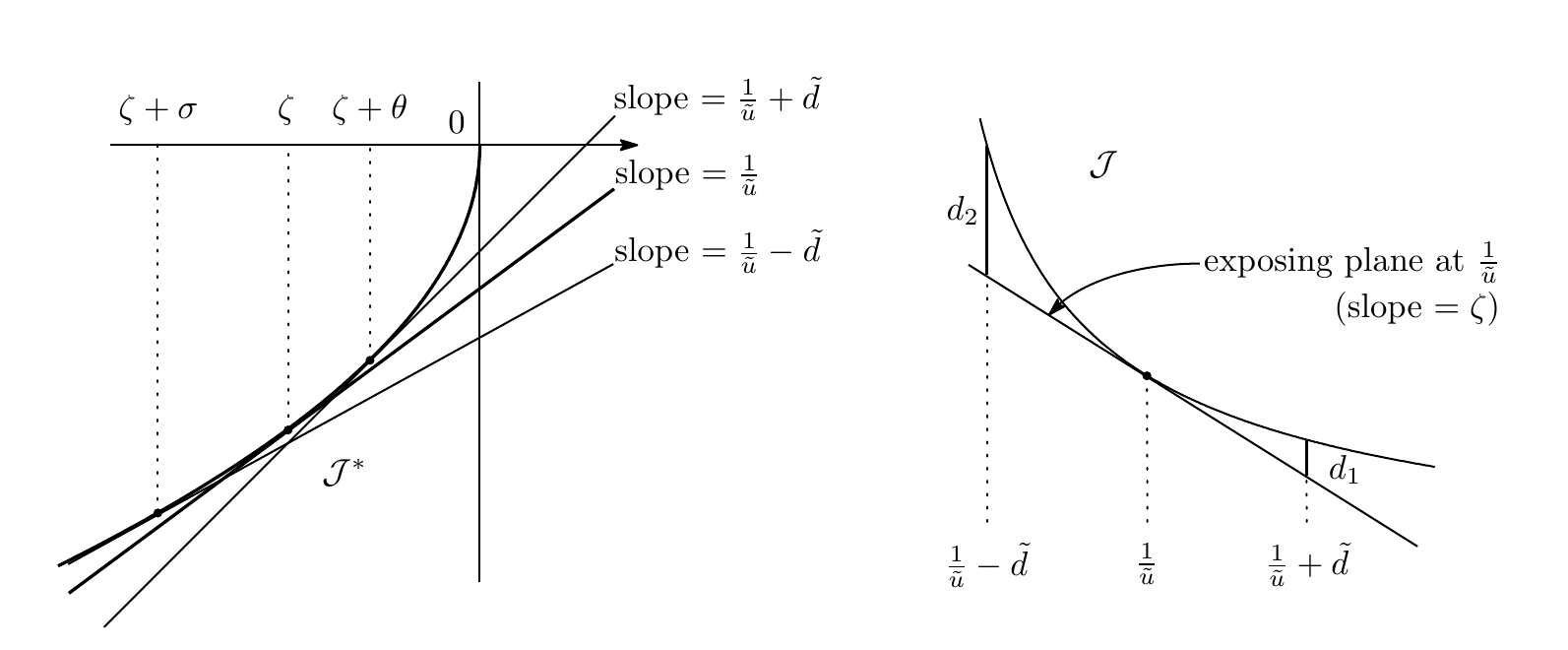}
\vspace{0.2cm}
\caption{\emph{Left}: The strict convexity of $\mc{J}^*$ shows that for a line of slope $1/\tilde{u} 
+ \tilde{d}$ the maximum is attained at $\zeta + \theta>\zeta$, while for a line of slope 
$1/\tilde{u} - \tilde{d}$ the maximum is attained at $\zeta + \sigma<\zeta$ (see~\eqref{e:theta} 
and~\eqref{e:sigma}, respectively).  \emph{Right}: The exposing plane condition shows that 
$d_1$ in~\eqref{e:expplus} and $d_2$ in~\eqref{e:expminus} are strictly positive.}
\label{fig:exposing_J_J*}
\end{center}
\end{figure}

By~\eqref{eq:exposing} with $y = \frac{1}{\tilde{u}} + \tilde{d}$, we find that
\begin{equation}
\begin{aligned}\label{e:expplus}
&- \gt \left(\frac{1}{\tilde{u}} + \tilde{d}\right) + \mc{J}^*(\gt + \zeta) - \mc{J}^*(\zeta)\\
&= - (\gt + \zeta) \left(\frac{1}{\tilde{u}} + \tilde{d}\right) + \mc{J}^*(\gt + \zeta) - \mc{J}^*(\zeta) 
+ \zeta \left(\frac{1}{\tilde{u}} + \tilde{d}\right)\\
&= - \mc{J}\left(\frac{1}{\tilde{u}} + \tilde{d}\right) + \zeta \left(\frac{1}{\tilde{u}} + \tilde{d}\right) - \mc{J}^*(\zeta)\\
&\leq -\mc{J}\left(\frac{1}{\tilde{u}} + \tilde{d}\right) + \zeta \left(\frac{1}{\tilde{u}} + \tilde{d}\right) 
- \zeta \frac{1}{\tilde{u}} + \mc{J}\left(\frac{1}{\tilde{u}}\right) = - d_1 < 0
\end{aligned}
\end{equation}
(see Fig.~\ref{fig:exposing_J_J*}).  On the set $A_I = \chv{\omega \colon \abs{O(\go,m,\gt+\zeta) 
- O(\go,m,\zeta)} < \tfrac12 \abs{d_1}}$, we have
\begin{equation} 
\label{e:BI}
P^{\omega,\zeta,m}_0 \prt{Y > \frac{1}{\tilde{u}} + \tilde{d}} < e^{- m \frac{\abs{d_1}}{2}} \to 0.
\end{equation}

\paragraph{Bound for $II$.}

Again, for $\gs < 0$,
\begin{equation}
\begin{aligned}
&P^{\omega,\zeta,m}_0 \prt{Y < \frac{1}{\tilde{u}} - \tilde{d}} \leq e^{- m \gs (\frac{1}{\tilde{u}} - \tilde{d})} 
\, E^{\go,\zeta,m}_0 \crt{e^{\gs Y}}\\
&= \exp \crl{m \crt{- \gs \left(\frac{1}{\tilde{u}} - \tilde{d}\right) + \frac{1}{m} 
\prt{\log E^\go_0 \crt{e^{(\gs + \zeta) H_m}} - \log E^\go_0 \crt{e^{\zeta H_m}}}}}\\
&\leq \exp \crl{m \crt{- \gs \left(\frac{1}{\tilde{u}} - \tilde{d}\right) + \mc{J}^*(\gs + \zeta) - \mc{J}^*(\zeta) 
- O(\go,m,\gs+\zeta) + O(\go,m,\zeta)}}.
\end{aligned}
\end{equation}
By the strict convexity of $\mc{J}$ and the fact that $\frac{1}{\tilde{u}} - \tilde{d} < \frac{1}{\tilde{u}}$, 
we can pick $\gs < 0$ (see Fig~\ref{fig:exposing_J_J*}) such that
\begin{equation} \label{e:sigma}
(\gs + \zeta) \left(\frac{1}{\tilde{u}} - \tilde{d}\right) - \mc{J}\left(\frac{1}{\tilde{u}} - \tilde{c}\right)
 = \mc{J}^*(\gs + \zeta).
\end{equation}
Similarly to~\eqref{e:expplus}, using~\eqref{eq:exposing} with $y = \frac{1}{\tilde{u}} - \tilde{d}$, we obtain
\begin{equation}\label{e:expminus}
\begin{aligned}
&- \gs \left(\frac{1}{\tilde{u}} - \tilde{d}\right) + \mc{J}^*(\gs + \zeta) - \mc{J}^*(\zeta) = - d_2 < 0,
\end{aligned}
\end{equation}
see Fig~\ref{fig:exposing_J_J*}.  For $\omega \in A_{II} = \chv{\omega \colon \abs{O(\go,m,\gs+\zeta) 
- O(\go,m,\zeta)} < \frac12\abs{d_2}}$
\begin{equation} 
\label{e:BII}
P^{\omega,\zeta,m}_0 \prt{Y < \frac{1}{\tilde{u}} - \tilde{d}} < e^{- m \frac12 \abs{d_2}} \to 0.
\end{equation}

\paragraph{Conclusion.}

For $n$ large enough, using~\eqref{e:BI} and~\eqref{e:BII} we see that
\begin{equation}
\omega \in A_I \cap A_{II} \quad \Longrightarrow \quad 
P^{\omega,\zeta,m}_0 (B_{\tilde{d}}(1/\tilde{u})^c) < \tfrac12 
\quad \Longrightarrow \quad 
P^{\omega,\zeta,m}_0 (B_{\tilde{d}}(1/\tilde{u})) \geq \tfrac12,
\end{equation}
and therefore, for large $n$, we conclude that
\begin{equation}
\abs{\frac{1}{n} \log P^{\omega,\zeta,m}_0 (B_{\tilde{d}}(1/\tilde{u}))} 
< \frac{1}{n}\log 2 < \tfrac14\gep.
\end{equation}
Hence
\begin{equation}
\begin{aligned}
&\mu \prt{\go \colon \frac{1}{n} \log P^{\omega,\zeta,m}_0 (B_{\tilde{d}}(1/\tilde{u})) 
< - \tfrac14\gep} \leq \mu \prt{A_I^c} + \mu \prt{A_{II}^c}.
\end{aligned}
\end{equation}
To complete the proof, note that Proposition~\ref{L:ZT} implies
\begin{equation}
\label{final}
\mu\prt{A_I^c} + \mu \prt{A_{II}^c} \leq C \, e^{- c n^{1-\delta}}.
\end{equation}
It is worth mentioning that the constant in~\eqref{final} does not depends on $n$. For fixed $u$ 
we choose $\gep>0$, and ~\eqref{e:halfI} together with~\eqref{e:convd} gives us $d,\tilde{d}$ 
and $\tilde{u}$.  After that, $d_1,d_2$ are given by the exposing plane conditions at the boundary 
of the ball of radius $\tilde{d}$ centered at $\frac{1}{\tilde{u}}$ (see~\eqref{e:expplus} and~\eqref{e:expminus}). 
Thus, even though the constant in~\eqref{final} depends on $d_1,d_2$, the latter are functions of 
$u$ and $\gep$ only, and not of $n$. The latter estimate, together with the bounds in~\eqref{last1}
and~\eqref{last2}, yield
\begin{equation} 
\label{e:LB*}
\mu \prt{\go \colon \frac{1}{n} \log P^\go_0 \prt{\frac{Z_n}{n} \geq u} 
+ \mc{I}(u) <- \gep} < C \, e^{- c n^{1-2\delta}}.
\end{equation}
Therefore~\eqref{e:BSI+} in Lemma~\ref{L:*} follows from~\eqref{e:UB*} and~\eqref{e:LB*}.
\end{proof}

\section{Proofs of SLLN and LDP} \label{SLLDP}

\subsection{Proof of Theorem~\ref{SLLN}}

\begin{proof}
We will prove the SLLN under the annealed law.  After that we get Theorem~\ref{SLLN} by 
noting that, for any event $A$, if $P^{\mu,\tau}_0(A)  = 1$, then, for $\mu^{\bb{N}}$-a.e.\ $\Omega$, 
$P^{\Omega,\tau}_0(A) = 1$, and by taking $A = \chv{\lim_{n\to\infty} \frac{X_n}{n} = v_\mu}$ we get
the claim.

To prove the SLLN under the annealed law, we will use Theorem~\ref{CEThm}. To this aim, let 
$\bb{P}$ be the joint law of doubly indexed variables  $\psi^{(k)}_n$ that are pair-wise independent 
in $k$ and such that, for each $k$, $\psi^{(k)}_n$ has law $P^\mu_0(Z_n=\cdot)$. From~\eqref{equiv_distr} 
we see that $\psi^{(k)}_{T_k}$ is distributed as $Y_k$ and $\psi^{\left(\ell(n)\right)}_{\bar{T}^n}$ is 
distributed as $\bar{Y}^n$.

Assumptions $(A1)$ and $(A2)$ are trivially satisfied. It remains to check $(A3)$ with $L = v_\mu$, 
for which we use the annealed large deviation estimates for RWRE. In fact, from Proposition~\ref{annLDP} 
we get
\begin{equation}
\label{e:expdecay}
\limsup_{n\to \infty} \frac{1}{n} \log P^\mu_0\prt{ \abs{Z_n/n - v_\mu}\geq\gep}
 \leq -\mc{I}(v_\mu +\gep)\vee -\mc{I}(v_\mu-\gep).
\end{equation}
In the zero-speed case, since $-\mc{I}(v_\mu +\gep) \vee -\mc{I}(v_\mu -\gep)<0$, the speed of decay is 
exponential in $n$ and $(A3)$ holds. In the positive-speed case, $\mc{I}(v_\mu -\gep) = 0$ and the 
bound in~\eqref{e:expdecay} is not useful. However, Proposition~\ref{flatannLDbound} yields 
the refined bound in the flat piece:
\begin{equation}
\label{e:poldecay}
 \limsup_{n\to \infty} \frac{1}{\log n} \log P^\mu_0\prt{ \abs{Z_n/n } <v_\mu -\gep} \leq 1-s.
\end{equation}
This speed of decay is polynomial in $n$, and hence $(A3)$ still holds for $\gd \in(0, s - 1)$.
\end{proof}

\subsection{Proof of Theorem~\ref{thm:LDP}}\label{LDPproof}

\begin{proof}
We start by showing that, for $\mu^{\bb{N}}$-a.e.\ $\Omega$, $P^{\Omega,\tau}(X_n/n\in \cdot)$ 
satisfies the LDP with rate $n$ and with some convex rate function, which we denote by 
$\mc{I}_{\mu,\tau}$. This follows from Campos \emph{et al.}~\cite[Theorem 1.2]{CDRRS13}, 
which states that any uniformly elliptic nearest-neighbour random walk in a dynamic random 
environment on $\Z$, for which the space-time-shift operator is ergodic, satisfies a quenched 
LDP with rate $n$ and with a convex rate function. In our case, the uniform ellipticity assumption 
is given in~\eqref{UEllip}. For the above mentioned ergodicity we note that, in view of the i.i.d.\
property of $\mu$, it suffices to establish ergodicity of the time-shift operator. The latter is  
true because the time-shift operator is actually strongly mixing w.r.t.\ to the law induced on the 
space of dynamic environments. In fact, translations of cylinder events over time eventually 
become separated by some resampling time, and therefore they are independent.

Having established the existence of a convex rate function $\mc{I}_{\mu,\tau}$, we can invoke 
Varadhan's Lemma (see e.g.~\cite[Chaper III]{dH00}) to guarantee that
\begin{equation} 
\label{CGF-cQRF}
\exists\lim_{n\to\infty} \frac{1}{n} \log E^{\gO,\tau}_0 \crt{e^{\gl X_n}} 
= \sup_{x \in \R} \big[\gl x - \mc{I}_{\mu,\tau}(x)\big]
\qquad \forall\, \lambda \in \R.
\end{equation}
On the other hand, we next show by means of Theorem~\ref{CEThm} that 
\begin{equation} 
\label{2Leg}
\lim_{n\to\infty} \frac{1}{n} \log E^{\gO,\tau}_0 \crt{e^{\gl X_n}} =\sup_{x \in \R} \big[\gl x - \mc{I}(x)\big]
\qquad \forall\, \lambda \in \R,
\end{equation}
with $\mc{I}$ from Proposition~\ref{staticLDPque}. Hence, combining~\eqref{CGF-cQRF} and~\eqref{2Leg}, 
and using the convexity of the rate functions involved, we obtain the identity $\mc{I}_{\mu,\tau}= \mc{I}$, 
which concludes the proof.

Let us finally check that indeed~\eqref{2Leg} holds. Note first that
\begin{equation}
\frac{1}{n} \log E^{\gO,\tau}_0 \crt{e^{\gl X_n}} 
= \frac{1}{n} \sum_{k = 1}^{\ell(n)-1} \log E^{\go_k}_0 \crt{e^{\gl Z_{T_k}}} 
+ \frac{1}{n} \log E^{\go_{\ell(n)}}_0 \crt{e^{\gl Z_{\bar{T}^n}}}.
\end{equation}
Now, let $\bb{P}$ be the law induced by the doubly indexed variables $\psi^{(k)}_n(\Omega) :=\log E^{\go_k}_0 
\crt{e^{\gl Z_{n}}}$ under $\mu^{\bb{N}}$. Then 
\begin{equation}
\psi^{(k)}_{T_k} = \log E^{\go_k}_0 \crt{e^{\gl Z_{T_k}}}, \qquad
\psi^{\left(\ell(n)\right)}_{\bar{T}^n} = \log E^\go_0 \crt{e^{\gl Z_{\bar{T}^n}}}.
\end{equation}
Again, Assumptions $(A1)$ and $(A2)$ of Theorem~\ref{CEThm} are readily satisfied, and it remains 
to check (A3) with $L = \mc{I}^*(\lambda)$. Finally, Theorem~\ref{P:MB} yields the latter assumption 
and gives the desired result.
\end{proof}


\end{document}